\documentclass[fms,times]{cuparticle}

\usepackage{amssymb,amsfonts,bbm,graphicx}
\usepackage{amsmath}
\usepackage{hyperref}

\newtheorem{theorem}{Theorem}[section]
\newdefinition{definition}[theorem]{Definition}
\newdefinition{remark}[theorem]{Remark}
\newdefinition{example}[theorem]{Example}

\newproof{proof}{Proof}

\volume{}
\doi{}

\newtheorem{lemma}[theorem]{Lemma}
\newtheorem{propo}[theorem]{Proposition}
\newtheorem{claim}[theorem]{Claim}

\newtheorem{conjecture}[theorem]{Conjecture}
\newtheorem{corollary}[theorem]{Corollary}

 \newcommand{\GG}{\mathcal{G}}

\newcommand{\len}[1]{|#1|}
\newcommand{\tough}{stubborn}
\newcommand{\TT}{S} 
\newcommand{\TM}{N} 
\newcommand{\TE}{E} 
\newcommand{\TB}{L} 
\newcommand{\TG}{R} 
\newcommand{\cc}{r_3} 
\newcommand{\muC}{\mu^{\#}} 
\newcommand{\muN}[1]{\mu(#1)}
\newcommand{\muCN}[1]{\muC(#1)}
\newcommand{\muCNe}[2]{e(#1,#2)}

\newcommand{\B}[1]{{\bf #1}}
\newcommand{\C}[1]{{\cal #1}}
\newcommand{\I}[1]{\mathbbm{#1}}

\newcommand{\dist}{\mathrm{dist}}
\renewcommand{\mid}{:}
\newcommand{\dd}{\,\mathrm{d}}
\newcommand{\hide}[1]{}
\renewcommand{\dots}{\hspace{0.9pt}.\hspace{0.3pt}.\hspace{0.3pt}.\hspace{1.5pt}}
\renewcommand{\ge}{\geqslant}
\renewcommand{\le}{\leqslant}
\renewcommand{\geq}{\geqslant}
\renewcommand{\leq}{\leqslant}

\begin{document}

\authorheadline{E.\ Cs\'oka, G.\ Lippner and O.\ Pikhurko}
\runningtitle{K\H{o}nig's Line Coloring and Vizing's Theorems for Graphings}
 
\begin{frontmatter}

\title{K\H{o}nig's Line Coloring and Vizing's Theorems for Graphings}
\author[1]{Endre Cs\'oka}
\author[2]{G\'abor Lippner}
\author[3]{Oleg Pikhurko}
\address[1]{MTA Alfr\'ed R\'enyi Institute of Mathematics, Hungarian Academy of 
Sciences, Budapest H-1053, Hungary
 \ead{csokaendre@gmail.com}}
\address[2]{Department of Mathematics, Northeastern University, Boston, MA 
02115, USA
\ead{g.lippner@neu.edu}}
\address[3]{Mathematics Institute and DIMAP,
University of Warwick,
Coventry CV4 7AL, UK
\ead{O.Pikhurko@warwick.ac.uk}}

\received{11 November 2015}

\begin{abstract}
The classical theorem of Vizing states that
every graph of maximum degree $d$ admits an edge-coloring with at 
most $d+1$ colors. Furthermore, as it was earlier shown by K\H{o}nig,
$d$ colors suffice if the graph is bipartite.

We investigate the existence of measurable edge-colorings for graphings (or
measure-preserving graphs).  A graphing is an analytic generalization of a 
bounded-degree graph that appears in various areas, such as sparse graph 
limits, orbit equivalence and measurable
group theory.  
We show that every graphing of maximum degree $d$ admits a measurable edge-coloring with $d + O(\sqrt{d})$ colors; furthermore, 
if the graphing has no odd cycles, then $d+1$ colors suffice. In fact,
if a certain conjecture about finite graphs that strengthens Vizing's theorem is true, then
our method will show that $d+1$ colors are always enough.
\end{abstract}

	\MSC[2010]{05C15, 05C63 (primary); 03E05, 03E15, 22F10, 28D05, 37A15 (secondary)}

 \end{frontmatter}

\section{Introduction}

The old theorem of K\H onig \cite{konig:16} states that  a bipartite graph
of maximum degree $d$ admits an edge-coloring with $d$ colors. (Here, all edge-colorings
are assumed to be \emph{proper}, that is, no two adjacent edges have the same color.) Some 50 years later,
Vizing~\cite{vizing:64} and, independently, Gupta~\cite{gupta:66} proved that, if we do not require
that the graph is bipartite, then $d+1$ colors suffice. 
These results laid the foundation of edge-coloring, an important and active area of graph theory; see, for example, 
the recent book of Stiebitz, Scheide, Toft and Favrholdt~\cite{stiebitz+scheide+toft+favrholdt:gec}.

In this paper, we consider measurable edge-colorings of  \emph{graphings} (which are graphs with some extra analytic structure, to be defined shortly). Although the graphs that we will consider may have 
infinitely many (typically, continuum many) vertices, we will always require that the maximum degree is bounded. 

If one does not impose any further structure, then K\H{o}nig's and Vizing's theorems extend, with the same bounds, to infinite graphs by the Axiom of Choice. Indeed, every finite subgraph is edge-colorable by the original 
theorem so the Compactness Principle gives the required edge-coloring of the whole graph. 

The first step towards graphings is to add Borel structure. Namely, a \emph{Borel graph} (see e.g.\ Lov\'asz \cite[Section 18.1]{lovasz:lngl}) is a triple $\GG=(V,\C B,E)$, where $(V,\C B)$ is a standard Borel space and
$E$ is a Borel subset of $V\times V$ that defines a symmetric and anti-reflexive binary relation. As we have
already mentioned, here we restrict ourselves to those graphs $\GG$ for which
the \emph{maximum degree} 
 $$
 \Delta(\GG):=\max\{\deg(x): x\in V\}
 $$
  is finite. While this definition
sounds rather abstract, it has found concrete applications to finite graphs: e.g.\ Elek and Lippner~\cite{elek+lippner:10} used Borel matchings to give another proof of the result of Nguyen and Onak~\cite{nguyen+onak:08} that the matching ratio in bounded-degree graphs is testable.

Define the \emph{Borel chromatic number} $\chi_{\C B}(\GG)$ of a Borel graph $\GG$ to be the minimum $k\in\I N$ such that  there is a partition $V=V_1\cup\dots\cup V_k$
into Borel \emph{independent} sets (that is, sets that do not span an edge of $E$). Also, the \emph{Borel chromatic index} $\chi'_{\C B}(\GG)$ is  the smallest number of Borel matchings that partition~$E$. (By a 
\emph{matching} we understand a set of pairwise disjoint edges; we do not require that every vertex is covered.) A systematic study of Borel colorings was initiated by Kechris, Solecki and Todorcevic \cite{kechris+solecki+todorcevic:99} who,
in particular, proved the following result.

\begin{theorem}[Kechris, Solecki and Todorcevic~\cite{kechris+solecki+todorcevic:99}]\label{th:KST} For every Borel graph $\GG$ of maximum degree
$d$, we have that $\chi_{\C B}(\GG)\le d+1$ and $\chi'_{\C B}(\GG)\le  2d-1$.\qed
 \end{theorem}

Very recently, Marks~\cite{marks:16} constructed, for every $d\ge 3$, an example of a $d$-regular Borel graph $\GG$ such that $\GG$ has no cycles, $\chi_{\C B}(\GG)=2$ and $\chi'_{\C B}(\GG)=2d-1$. (Such a graph for $d=2$ was earlier
constructed by Laczkovich~\cite{laczkovich:88}.)  We see that the
Borel chromatic index may behave very differently from the finite case. 

Marks~\cite{marks:16} also considered the version of the problem when, 
additionally, we have a measure $\mu$ on $(V,\C B)$ and ask for the \emph{measurable chromatic
index} $\chi'_{\mu}(\C B)$, the smallest integer $k$ for which there is a Borel 
partition $E=E_0\cup E_1\cup\dots\cup E_k$ such that $E_i$ is a matching for each $i\in [k]:=\{1,\dots,k\}$
while the set of vertices covered by $E_0$ has measure zero. In particular, Marks~\cite[Question 4.9]{marks:16} asked if  
 \begin{equation}\label{eq:MarksQn}
 \chi'_{\mu}(\GG)\le \Delta(\GG)+1
 \end{equation} 
 always holds and proved~\cite[Theorem 4.8]{marks:16} that this is the case for $\Delta(\GG)=3$. (It is not hard to show that 
(\ref{eq:MarksQn}) holds when $\Delta(\GG)\le 2$.)

Although we cannot answer the original question of Marks, we can improve the upper bound
on the measurable chromatic index when the measure $\mu$ defines a \emph{graphing} (or equivalently if the measure $\mu$ is invariant, see~\cite[Section~2]{kechris+miller:toe} for definitions and proofs). We
believe that this is a very important case of Marks' question since measure-preserving
systems are central to many areas of mathematics. (In fact, the question whether~\eqref{eq:MarksQn} holds for graphings was earlier asked by Ab\'ert~\cite[Question~35]{abert:10:questions}.)

\begin{definition}\label{df:graphing} A  \emph{graphing} (or a \emph{measure-preserving graph}) is a quadruple  $\GG = (V,\C B, E,\mu)$, where $(V,\C B)$ is a standard Borel space, $\mu$ is a probability measure on $(V,\C B)$, and there are finitely many triples $(\phi_1,A_1,B_1),\dots,(\phi_k,A_k,B_k)$ such 
that 
 \begin{equation}\label{eq:E}
 E=\big\{\, (x,y)\mid x\not=y\ \ \&\ \ \exists\, i\in [k]\ \ \phi_i(x)=y
 \mbox{ or }\phi_i(y)=x\big\}\subseteq V\times V
 \end{equation}
 and each $\phi_i$ is an invertible Borel bijection between $A_i,B_i\in \C B$ that preserves 
the measure $\mu$.
 
\end{definition}

\begin{remark}  Note that if $(V,\C B, E,\mu)$ is a graphing,
then $(V,\C B, E)$ is a Borel graph.\end{remark}

We refer the reader to Lov\'asz~\cite[Section 18.2]{lovasz:lngl} for an introduction to graphings. There are other equivalent definitions. We chose
the above definition as it is combinatorial in nature and allows us to derive other properties of
graphings rather easily. While our use of the term \emph{graphing} seems to be
standard in the area of sparse graph limits, it has another meaning in descriptive set theory and orbit equivalence.

Graphings appear in various fields.
One can view $(V,\C B,\mu,\phi_1,\dots,\phi_k)$ as a generalization of a dynamical measure-preserving system. When we pass to the graphing $\GG$, we lose some information but many properties (such as ergodicity) can still be recovered. Also, if $\phi_1,\dots,\phi_k$ come from a measure-preserving group action (with $A_i=B_i=V$), then the connectivity components of $\GG$
correspond to orbits. Indeed, measure-preserving graphs play an important role in orbit equivalence and measurable group theory (see \cite{gaboriau:02,gaboriau:10,kechris:gaega,kechris+miller:toe,shalom:05}). 
For example, the well-known \emph{Fixed Price Problem} for groups (introduced by Levitt~\cite{levitt:95}
and extensively studied in e.g.~\cite{gaboriau:98,gaboriau:00})
involves finding the infimum
of  the average degree $\int_V \deg(x)\dd\mu(x)$ over all measure-preserving graphs on $(V,\C B,\mu)$ with the given
connectivity components. Measure-preserving graphs are also of interest in descriptive set theory (see~\cite{kechris+marks:survey}).
We came to this topic motivated by limits of bounded-degree graphs 
(see e.g.~\cite{lovasz:lngl}) since
graphings can be used to represent a limit object for both the Benjamini-Schramm~\cite{benjamini+schramm:01} (or local) convergence and the Bollob\'as-Riordan~\cite{bollobas+riordan:11} (or global-local)
convergence, as shown by Aldous and Lyons~\cite{aldous+lyons:07}, Elek~\cite{elek:07} and
Hatami, Lov\'asz and Szegedy~\cite{hatami+szegedy+lovasz:14}.

\begin{example}\label{ex:translation} Given $\alpha\in\I R$, let $\C T_\alpha$
be the graphing on the real unit interval $([0,1),\C B,\lambda)$, with the 
Lebesgue measure $\lambda$, generated by the \emph{$\alpha$-translation} 
$t_\alpha:[0,1)\to[0,1)$ that maps $x$ to $x+\alpha\pmod 1$.\end{example}

The above simple example of a graphing exhibits various interesting properties that contradict ``finite intuition'' when $\alpha$ is irrational. Namely, $E$ defines a 2-regular and acyclic graph while the ergodicity of $t_\alpha\circ t_\alpha=t_{2\alpha}$ implies that every Borel vertex 2-coloring or every Borel matching misses a set of vertices of positive measure (and thus
each of $\chi_{\C B}(\C T_\alpha)$, $\chi'_{\C B}(\C T_\alpha)$ and $\chi_\lambda'(\C T_\alpha)$ is strictly larger than~$2$). In particular, we see that the property of being \emph{bipartite} (that is, $\chi_{\C B}(\GG)\le 2$) may be strictly stronger than having no odd cycles.

We can make a finite graph $G=(V,E)$ into a graphing by letting $\C B=2^V$ consist of all subsets of $V$
and $\mu$ be the uniform measure on $V$.  Here, the smallest $k$ that satisfies (\ref{eq:E}) is equal to the
minimum number of graphs with degree bound 2 that decompose $E$.
This is trivially at least $\Delta(G)/2$ and,  by Vizing's theorem, is at most $\lceil (\Delta(G) + 1) / 2 \rceil$. 
Also, if we additionally require that $A_i\cap B_i=\emptyset$ for all $i\in [k]$, then the smallest $k$ is exactly the chromatic index~$\chi'(G)$.  In Section~\ref{MinK} we consider the smallest
$k$ in Definition~\ref{df:graphing} that suffices for every graphing of degree bound $d$ as well as its
variant where a null-set of errors is allowed. It should not be surprising to the reader that Borel and measurable chromatic indices play an important role in estimating these parameters.
This provides further motivation for our  main result that $\chi_\mu'(\GG)=(1+o(1))\,\Delta(\GG)$ for every graphing $\GG$:

\begin{theorem}\label{th:main} If $\GG=(V,\C B,E,\mu)$ is a graphing with maximum degree at most $d$, then its measurable chromatic index is at most $d+O(\sqrt{d})$. Moreover, if $\GG$ has no odd cycles, then $\chi'_\mu(\GG)\le d+1$.
\end{theorem}

In fact, Theorem~\ref{th:main} is a direct consequence of Lemma~\ref{pre-colored} and Theorem~\ref{vizing theorem}. In order to state them, we need some further preparation. 

\begin{definition}\label{df:f}
Let $f(k)$ be the smallest integer $f$ such that for every $d\in[k]$ the following holds. Let $G$ be an 
arbitrary finite
graph such that every degree is at most $d$, except at most one vertex of degree $d+1$. Suppose that at most $d-1$ \emph{leaves} (that is, edges with one of their endpoints having degree 1) are pre-colored. Then this pre-coloring can be extended to an edge-coloring of the whole graph
$G$ that uses at most $d+f$ different colors. 
 \end{definition}

By definition, the function $f(k)$ is non-decreasing in $k$. Since we allow a vertex of degree $k+1$ (when $d=k$), we have that $f(k)\ge 1$.  We make the following conjecture which, if true, will give a strengthening of Vizing's theorem. 

\begin{conjecture}\label{cj:f} $f(k)=1$ for all $k\ge 1$.\end{conjecture}

Conjecture~\ref{cj:f} trivially holds for $k\le 2$. Bal\'azs Udvari (personal communication) proved it
for $k=3$ but his proof does not seem to extend to larger~$k$.  We note that allowing a vertex of degree $d+1$ 
seems to be not an essential extension, but the pre-colored edges cause the difficulties. 
For general $k$, we can prove a weaker bound $f(k)=O(\sqrt k)$, which follows from the following lemma.

\begin{lemma}\label{pre-colored} Let $d$ be sufficiently large. Then every pre-coloring of at most $d$ leaves of a finite graph $G$ with $\Delta(G)\le d$ extends to an edge-coloring of $G$ that uses 
at most $d+ 9\sqrt d$ colors.
\end{lemma}

The function $f$ is of interest because of the following relation to the measurable chromatic index of graphings given by Theorem~\ref{vizing theorem}. Let us call a set $X$ of vertices (in a finite or infinite graph) \emph{$r$-sparse} if for every distinct $x,y\in X$ 
the graph distance between $x$ and $y$ is strictly larger than $r$. For example, a set is $1$-sparse if
and only if it is independent.

\begin{theorem}\label{vizing theorem} For every $d \geq 1$ there is $r_0 = r_0(d)$ such that if $\GG=(V,\C B,E,\mu)$ is a graphing with maximum degree at most $d+1$ such that the set $J$ of vertices of degree $d+1$ is $r_0$-sparse, then $\chi_\mu'(\GG)\le d+f(d)$. If, furthermore, $\GG$ has no odd cycles, then
$\chi_\mu'(\GG)\le d+1$.
\end{theorem}

\begin{remark}\label{rm:evend} Laczkovich \cite{laczkovich:88} for $d=2$ and Conley and Kechris~\cite[Section~6]{conley+kechris:13}
for every even $d\ge 4$ proved that  there exists a bipartite $d$-regular graphing $\GG$ such
that every Borel matching misses a set of vertices of positive measure. Hence, $d+1$ colors are 
necessary in Theorem~\ref{th:main} for such $d$, even in the bipartite case.
If Conjecture~\ref{cj:f} is true, then $d+1$ colors always suffice.
\end{remark}

This paper is organized as follows. Section~\ref{notation} collects some frequently used notation.
Basic properties of graphings that are needed in the proofs are discussed in Section~\ref{basic}.  
Section~\ref{MainInduction} formally describes the main inductive step (roughly, removing 
a matching $M$ that covers high degree vertices) and how this yields Theorem~\ref{vizing theorem}. 
Section~\ref{sec:augmenting} shows how to construct the required matching $M$, provided there is 
a sequence of matchings $(M_i)_{i=0}^\infty$ that stabilizes ``fast''. 
The main bulk of the proof appears in 
Section~\ref{section:shortaugmenting} where we inductively construct $M_{i+1}$ by augmenting
$M_i$ along paths of length at most $2i+1$. The fast stabilization of $M_i$'s 
is derived from a variant of the expansion property. This is relatively straightforward for
the case when there are no odd cycles and is done in Section~\ref{NoOddCycles}. The remainder
of Section~\ref{section:shortaugmenting} deals with the general case.
Lemma~\ref{pre-colored} is proved in Section~\ref{se:pre-colored}.
An application of Theorem~\ref{vizing theorem} (on the minimum number of
maps that generate a given graphing) is presented in Section~\ref{MinK}.

When presenting the long and difficult proof of Theorem~\ref{vizing theorem},
we tried to split it into smaller steps. (For example, Theorem~\ref{vizing theorem} follows from Theorem~\ref{theorem:inductive_step}, which in turn follows from~Theorem~\ref{shortalternating_theorem}.) Hopefully, this
makes the proof easier to follow and understand.

\section{Some notation}\label{notation}

For reader's convenience, we collect various notation here, sometimes repeating definitions that appear elsewhere.  

Let $G=(V,E)$ be a graph. 
For $A,B\subseteq V$, the \emph{distance} $\dist(A,B)$ is the shortest length of a path connecting a vertex in $A$ to a vertex in $B$.  Also, $E(A,B):=E\cap (A\times B)$ denotes the set of adjacent pairs
$(a,b)$ with $a\in A$ and $b\in B$. Note that we take the ordered pairs, so
that $|E(A,B)|$ counts the edges inside  $A\cap B$ twice.  The \emph{complement} of $A\subseteq V$
is $A^c:=V\setminus A$.   The \emph{$k$-neighborhood} 
of $A$ is
 $$
 N_{k}(A):=\{x\in V\mid \dist(\{x\},A)\le k\}.
 $$
 Recall that the set $A$ is called \emph{$r$-sparse} if every two distinct 
vertices of $A$ are at distance strictly larger than~$r$. It is \emph{$r$-dense} if every vertex of $V$
is at distance at most $r$ from $A$ (equivalently if $N_{r}(A)=V$).
 The \emph{degree} $\deg(x)$ of $x\in V$ is the number of edges in $E$ 
 containing $x$. The \emph{maximum degree} is $\Delta(G):=\max\{\deg(x)\mid 
 x\in V\}$. 
For a set of edges $C\subseteq E$, let 
$V(C):=\cup_{(x,y)\in C} \{x,y\}$ consist of vertices that are covered by at least one edge of~$C$.

We may omit the set-defining brackets, for example, abbreviating
$N_{1}(\{x\})$ to $N_{1}(x)$. Also, we write $\I N:=\{0,1,2,\dots\}$ and $[k]:=\{1,\dots,k\}$.
When applying combinatorial arguments to $(V,E)$, we will usually view $E$ as a set of unordered pairs and write e.g.\ $\{x,y\}\in E$ to mean $(x,y),(y,x)\in E$.

For a path $p$, $\len{p}$ will denote its \emph{length}, i.e.\ the
number of edges in $p$.  The path $p$ is called \emph{even} (resp.\ \emph{odd}) if its length $\len{p}$ is even (resp.\ odd).

\section{Basic properties of graphings}\label{basic}

This section discusses various properties that every graphing $\C G=(V,\C B,E,\mu)$ as in
Definition~\ref{df:graphing} possesses. Their proofs can be found
in Sections 18.1--18.2 of Lov\'asz' book~\cite{lovasz:lngl}. In fact, some of
these facts are immediate consequences of the Uniformization Theorem of Lusin-Novikov, see e.g.~\cite[Theorem~18.10]{kechris:cdst}.

\hide{
One point of
view on Definition~\ref{df:graphing} is to consider $E$ as a finite union of Borel matchings $M_i=\{ \{x,y\}\mid x\not =y\ \&\ \phi_i(x)=y\}$, $i\in[k]$. Clearly, by shrinking the sets $A_i$ and $B_i$, one can additionally
require that $M_i\cap M_j=\emptyset$ for all $i\not=j$.
}%

Since each $\phi_i$ in Definition~\ref{df:graphing} is measure-preserving, we have that 
 \begin{equation}\label{eq:MeasPres}
 \int_{A} \deg_B(x)\dd\mu(x)=\int_B\deg_A(x)\dd\mu(x),\qquad\mbox{for all $A,B\in\C B$},
 \end{equation}
 where e.g.\ $\deg_A(x)$ is the number of edges that $x\in V$ sends to $A\in\C B$. (It readily follows
from Definition~\ref{df:graphing}  that the function $\deg_A:V\to\I N$ is Borel.) 
When we make a finite graph $(V,E)$ into a graphing on $|V|$ atoms, then (\ref{eq:MeasPres}) 
corresponds to the trivial fact that the number of edges between sets $A,B\subseteq V$ can be counted
either from $A$ or from~$B$.

Conversely, it is known (see \cite[Theorem~18.21]{lovasz:lngl}) that if a measure $\mu$ on a Borel graph
$(V,\C B,E)$ satisfies (\ref{eq:MeasPres}), then $\GG=(V,\C B,E,\mu)$ is a graphing. (In fact,
one can take  $k=2\Delta(\GG)-1$ in Definition~\ref{df:graphing}; not surprisingly,
Theorem~\ref{th:KST} is used in the proof.) 

If $A\subseteq V$ is Borel, then the set $N_1(A)\subseteq V$ is also Borel, as it is the union of $A$
and the Borel sets $\phi_{i}^{\sigma}(A):=\{x\mid \exists\, y\in A\ \phi_i^\sigma(y)=x\}$ over $i\in [k]$ and $\sigma\in\{-1,1\}$. Similarly,
it follows that  ``locally'' defined subsets of $V$, 
such as for example the set of vertices that belong to a
triangle, are Borel (\cite[Exercise 18.8]{lovasz:lngl}). 

Also,
$(V,\C B,E^m)$ is a Borel graph, where $E^m$ consists
of pairs of distinct vertices at distance at most $m$ in $E$. Indeed, 
$E^m$ can be represented as in~\eqref{eq:E} for finitely many Borel maps, each being
a composition of at most $m$ of the maps $\phi_1^{\pm1},\dots,\phi_k^{\pm1}$
restricted to the (Borel) set where this composition is defined.
Combining this observation with Theorem~\ref{th:KST}, we obtain the following useful corollary.

\begin{corollary}\label{cr:label} For every Borel graph $\GG=(V,\C B,E)$ and $k\in \I N$ there is a
\emph{$k$-sparse labeling}, that is, a Borel function $\ell:V\to [m]$ for some $m\in\I N$
such that each part $\ell^{-1}(i)$ is $k$-sparse.\qed\end{corollary}

In fact, in the above corollary it suffices to take $m=1+\Delta(\GG)\sum_{i=1}^{k} (\Delta(\GG)-1)^{i-1}$, the maximum possible 
size of the $k$-neighborhood of a vertex.

The following proposition (see \cite[Lemma~18.19]{lovasz:lngl}) implies that if we construct objects inside a graphing in a Borel way, then any subgraph that we encounter is still a graphing. This will be implicitly used many times here 
(e.g.\ when  we remove a Borel matching from a graphing).

\begin{propo}\label{pr:subgraphing} If 
$\GG=(V,\C B,E,\mu)$ is a graphing and $E'\subseteq E$ is a Borel symmetric subset, then $\GG'=(V,\C B,E',\mu)$ is a graphing.\end{propo}
 \begin{proof} Let measure-preserving maps $\phi_1,\dots,\phi_k$ represent $\GG$ as in Definition~\ref{df:graphing}. Then their appropriately defined restrictions $\phi_1',\dots,\phi_k'$ to $E'$ represent $\GG'$. One can show directly
 (or invoke a classical theorem of Lusin~\cite{lusin:30}) that the range and the domain
 of each $\phi_i'$ are Borel.\qed\end{proof}

\begin{lemma}\label{lm:finite} Let $(V,\C B,E)$ be a Borel graph of maximum degree at most $d+1$ such that
no two vertices of degree $d+1$ are adjacent.  Then we can edge-color
all finite connectivity components of $\GG$ in a Borel way, using at most $d+1$ colors. 
\end{lemma}
\begin{proof} 
For $i=1,2,\dots$, we color all components with exactly $i+1$ vertices. Given $i$,
fix an $i$-sparse labeling $\ell:V\to\I N$, which exists by Corollary~\ref{cr:label}. The labels in each component with $i+1$ vertices are all different. Choose an isomorphism-invariant rule how to edge-color each labeled component with colors from $[d+1]$. Note that at least one coloring exists by the extension of Vizing's theorem by Fournier~\cite{fournier:78} that $\chi'(G)\le \Delta(G)$ if no two vertices of maximum degree are adjacent
(see also Berge and Fournier~\cite{berge+fournier:91} for a short proof).
Apply this rule consistently everywhere. Each color class is Borel, as the countable union over $i$ of Borel sets.
\qed\end{proof}

One can define the measure $\muC$ on $(V\times V,\C B\times\C B)$ by 
stipulating that 
 $$
 \muC(A\times B)=\int_A \deg_B(x)\dd\mu(x),\qquad\mbox{ for $A,B\in\C B$}
 $$
 and extending $\muC$ to the product $\sigma$-algebra $\C B\times\C B$ by Caratheodory's theorem. 
It can be shown that $\muC((V\times V)\setminus E)=0$, see \cite[Lemma 18.14]{lovasz:lngl}. 
Thus, in other words, $\muC$ is the
product of $\mu$ with the counting measure, restricted to $E$. Property~(\ref{eq:MeasPres}) shows
that $\muC$ is symmetric: $\muC(A\times B)=\muC(B\times A)$. 
 \hide{An equivalent definition, once
we make the matchings $M_i$ disjoint, is to let
 $$
 \muC(X):=\sum_{i=1}^k \mu\big(x \in A_i\mid \{x,\phi_i(x)\}\in X\big),\qquad X\in\C B\times\C B.
 $$
 }

If $X\subseteq V$ has measure zero, then $Y=\{y\in V\mid \dist(y,X)<\infty\}$, the union of all connectivity components intersecting $X$, also has measure zero.  Indeed, $Y$ is  
the countable union of the images of the null-set $X$ by finite 
compositions of $\phi_1^{\pm1},\dots,\phi_k^{\pm 1}$, where the maps $\phi_i$ are as in Definition~\ref{df:graphing}. The analogous claim applies to any $\muC$-null-set $X\subseteq E$.  
We will implicitly use this in the proof of Theorem~\ref{vizing theorem}: whenever we encounter
some  null-set of ``errors'', we move all edges from connectivity components with errors 
to the  exceptional part $E_0\subseteq E$
from the definition of $\chi'_\mu(\GG)$.  

One could define yet another chromatic index $\chi_{\mu}^*(\GG)$, where \textbf{every} edge has to be colored but each color class is only a \emph{measurable} subset of $E$ (that is, belongs to the completion of $\C B\times\C B$ with respect to $\muC$). It is easy to come up with an example when $\chi_{\mu}^*$ is strictly larger
than $\chi_{\mu}'$ (e.g.\ add a null-set
of high-degree vertices). However, considering $\chi_{\mu}^*$ would give nothing new in the context of 
Theorem~\ref{vizing theorem} because we can repair any null-set of errors by recoloring  all components containing them via the Axiom of Choice. (Note that we need at most
$d+1\le d+f(d)$ colors by Fournier's theorem~\cite{fournier:78}.)  We
restrict ourselves to $\chi_{\mu}'$ for convenience, so that we can stay within the Borel universe
(namely, all sets that we will encounter in the proof of Theorem~\ref{vizing theorem} are Borel).


\section{The main induction}\label{MainInduction}

Before we proceed with the proof of Theorem~\ref{vizing theorem}, it may be instructive to mention why known proofs of Vizing's theorem do not seem to extend to graphings. These proofs proceed by
some induction, typically on $|E|$. When we extend the current edge-coloring to 
a new edge $\{x,y\}$, we may need to swap colors in some maximal 2-color path 
$p$ that starts with $x$ or $y$. Unfortunately, we do not have 
any control over the length of $p$. This causes a problem when we do countably many iterations  in a graphing (each time
swapping a Borel family of such paths in parallel) because the set of edges that flip their color
infinitely often may have positive measure.

On the other hand, K\H{o}nig's theorem for finite graphs can be proved with much less back-tracking: take any matching $M$ that covers all vertices of maximum degree, color it with a new color and apply induction to the remaining graph $\GG\setminus M$. 
We prove Theorem~\ref{vizing theorem} by a similar induction on $\Delta(\GG)$. 
The difficulty with this approach is that even a finite (non-bipartite) graph need not have a matching covering
all vertices of maximum degree. So instead we change the inductive assumption: $\GG$ has 
maximum degree at most $d$ except an  $r_0(d)$-sparse set of vertices of degree $d+1$,
where $r_0:\I N\to\I N$ is a fast-growing function. Thus we want to find a matching $M$ that
covers all vertices of degree $d+1$ and ``most'' vertices of degree $d$, so that $\GG\setminus M$ satisfies
the sparseness assumption for $d-1$. This may still be impossible. 
However, if we remove all so-called stumps (to be colored 
later using Lemma~\ref{pre-colored}) and, for some technical reasons, all finite components,  then the desired matching $M$ exists. 

In the rest of this section, we define what a stump is, state the main inductive step (Theorem~\ref{theorem:inductive_step}) and show how it implies Theorem~\ref{vizing theorem}.

\begin{definition} Given a graph $G = (V, E)$ and an integer $d$ such that
$\Delta(G)\le d+1$, we call a set $A\subseteq
V$ lying inside some infinite connectivity component $C$ of $G$ a \emph{stump} if the number of vertices in $A$ is finite,
$|A|\ge 2$, 
$|E(A,A^c)|\le d-1$, and every
vertex of $A$ has degree $d$ in $G$ except at most one vertex of degree $d+1$.
 \end{definition}
 

\begin{theorem}\label{theorem:inductive_step}
For every $d \geq 2$ and $r \geq 1$, there is $r_1=r_1(d,r)$ such that the following holds. Let $\GG=(V,\C B,E,\mu)$ be a graphing with degree bound $d+1$ that has no finite components.
Suppose also that $\GG$ has no odd cycles or has no stumps. If  the set $J\subseteq V$ 
of vertices of degree $d+1$ is $r_1$-sparse, then there is a Borel matching $M$ such that,
up to removing a null-set, $\GG\setminus M$ has maximum degree
at most $d$ and its set of degree-$d$ vertices is $r$-sparse.
\end{theorem}

Let us show how Theorem~\ref{theorem:inductive_step} implies Theorem~\ref{vizing theorem}.\medskip

\noindent{\sc Proof of Theorem~\ref{vizing theorem}}. We use induction on $d$. 
For $d=1$ it is true with $r_0(1) = 1$: each component can have at most three vertices
so the required 2-edge-coloring exists by Lemma~\ref{lm:finite}. 
(Note that $f(1)=1$.) 

Let $d\ge 2$. Let $r:=r_0(d-1)$ be the value returned by Theorem~\ref{vizing theorem} for $d-1$,
using the inductive assumption. Let $r_0:=r_1(d,r)$ be the value returned by Theorem~\ref{theorem:inductive_step} on input $(d,r)$. We claim
that this $r_0$ suffices.
Take any graphing $\GG$ as in Theorem~\ref{vizing theorem}.
Let $J$ denote the $r_0$-sparse set of vertices of degree $d+1$ in~$\GG$.

First, let us do the case when $\GG$ has no odd cycles.  Do \emph{clean-up}, that is, remove all finite components from $\GG$ and edge-color them with $d+1$ colors using Lemma~\ref{lm:finite} (whose assumptions are satisfied since $J$ is an independent set). Now, Theorem~\ref{theorem:inductive_step} gives a Borel matching $M$ such that, up to
removing a null-set, $\GG':=\GG\setminus M$ has no vertices of degree larger than $d$ while its degree-$d$
vertices form an $r$-sparse set. So, by induction, we can color $\GG'$ with $d$ colors, and using the last color for $M$ we get a Borel $(d+1)$-edge-coloring of $\GG$ a.e., as required.

In the general case, we additionally make sure that there are no stumps  before we apply Theorem~\ref{theorem:inductive_step}. Namely, for each integer $i\ge 1$
in the increasing order of $i$, we fix a $2i$-sparse labeling. 
For each isomorphism type of a labeled
stump that has exactly $i+1$ vertices and spans a connected subgraph, pick all such stumps $A$ in $\GG$ and remove all edges inside each $A$. (Note that we keep all edges between $A$ and its complement $A^c$.) 
Clearly, after we have removed a stump, all its vertices have degree at most $d-1$. In particular, none of them can belong to a stump now. Also, every two stumps that were removed simultaneously are 
vertex-disjoint since the labeling was sufficiently sparse. The final graphing has no stumps because for every stump $A$ there is a stump $A'\subseteq A$ that spans a connected subgraph (and our procedure considers $A'$ at some point). 

Having removed all stumps,  we do clean-up (that is, we remove and edge-color all finite components of the current graphing). Denote the remaining graphing by $\GG'$. It has degree bound $d+1$ and the set of vertices of degree $d+1$ is still  $r_0$-sparse in $\GG'$. But $\GG'$ has no stumps nor finite components, so we can apply Theorem~\ref{theorem:inductive_step} as above and inductively obtain a Borel edge-coloring of $\GG'$ a.e.\
with $d -1+ f(d-1) +1 = d+f(d-1)$ colors. 
It remains to color edges inside the stumps that we have removed. By the vertex-disjointness,
we can treat each stump independently of the others.
The colors on the at most $d-1$ 
edges that connect the stump to its complement are already assigned.  The definition of $f$ 
(Definition~\ref{df:f})  shows that this pre-coloring can be extended to a $(d+f(d))$-coloring of the whole 
stump. This again can be done in a Borel way, by applying some fixed rule consistently. Finally since $f(d) \geq f(d-1)$, we get a Borel $(d+f(d))$-edge-coloring of $\GG$ a.e., as desired.\qed\medskip


\section{Proof of Theorem~\ref{theorem:inductive_step}}\label{sec:augmenting}

Here we present the proof of Theorem~\ref{theorem:inductive_step}, by reducing it to
Theorem~\ref{shortalternating_theorem} that in turn will be proved in Section~\ref{section:shortaugmenting}.

It is known that a $d$-regular expander graphing that is bipartite or has no edge-cuts with fewer
than $d$ edges admits a measurable perfect matching. This has been shown by Lyons and Nazarov~\cite{lyons+nazarov:11} for the former case and Cs\'oka and Lippner~\cite{csoka+lippner}  for the latter case. 
What follows is an adaptation of these proofs to allow a sparse set of vertices of degree $d+1$. 
As we will see, these ``exceptional'' vertices do not
cause any considerable difficulties. The real problem is that
our graphing $\GG$ need not be an expander!  Probably, the most crucial observation of this paper is that we can make the graphing behave like an expander at the expense of designating a small set $K$ of vertices around each of which at most one error (an unmatched degree-$d$ vertex) is allowed. 
As we will see in Lemma~\ref{lemma:expansion}, if $K$ is $O(1)$-dense then 
$\mu(N_{1}(X))= (1+\Omega(1))\, \mu(X)$ for every $X\subseteq K^c$, that is, 
sets disjoint from $K$ expand in measure.  
Theorem~\ref{shortalternating_theorem} then shows that such expansion is enough 
to obtain the matching $M$ required in 
Theorem~\ref{theorem:inductive_step}.\medskip 

\noindent{\sc Proof of Theorem~\ref{theorem:inductive_step}.} Given $d$ and 
$r$, let $r':=3r+7$ and let $r_1$ be sufficiently large. Let $\GG=(V,\C 
B,E,\mu)$ satisfy all assumptions of Theorem~\ref{theorem:inductive_step}. In 
particular, the set $J$ of vertices of degree $d+1$ is $r_1$-sparse.

\paragraph{Constructing the set $K$:}
First, we construct a set $K \subseteq V$ of vertices of degree at most $d$ such that $J\cup K$ is 
$(r+2)$-sparse while $K$ is \emph{$r'$-dense}, meaning that for every $x\in V$ there is $y\in K$ with
$\dist(x,y)\le r'$. (This density requirement will later give us the desired expansion property.) 

Such a set $K$ can be constructed as follows. By Corollary~\ref{cr:label},
take an $(r+2)$-sparse labeling $\ell:V\to[m]$.
Next, iteratively for $i=1,\dots,m$, add to $K$ all vertices $x\in V$
such that $\ell(x)=i$ and $J\cup K\cup\{x\}$ is still $(r+2)$-sparse. By the definition of
$\ell$, no two vertices with the same label can create a conflict to the sparseness. Thus the set
$J\cup K$ remains $(r+2)$-sparse throughout the whole procedure.

Let us verify that the final set $K$ is $r'$-dense. Take any $y\not\in K$. Since we did not add $y$, it is at distance at most $r+2$ from $J\cup K$. Assume that $\dist(y,K)>r+2$ as otherwise we are done. Then
there is $z\in J$ with $\dist(y,z)\le r+2$. By the assumptions of Theorem~\ref{theorem:inductive_step}, all 
connectivity components (in particular the component of $z$) are infinite.
Let $y'$ be a vertex at distance exactly $r+3$ from~$z$. Such a vertex exists:
take any walk that starts at $z$ and eventually goes away from it and let $y'$ 
be the first visited vertex that is at distance at least $r+3$ (and thus 
exactly $r+3$) from $z$. This vertex $y'$ is at distance at least $r_1-(r+3)> 
r+2$ 
from $J\setminus\{z\}$ by the $r_1$-sparseness of $J$. By the construction of $K$, we have 
that $\dist(y',K)\le r+2$. By the triangle inequality,
 \begin{eqnarray*}
 \dist(y,K)&\le& \dist(y,z)+\dist(z,y')+\dist(y',K)\\
  &\le &(r+2)+(r+3)+(r+2)\ =\ r'.
  \end{eqnarray*}
 Thus $K$ is indeed $r'$-dense.

\paragraph{Stars of exceptional vertices:}
For any $x \in K$ let us define the \textit{star} of $x$ to be the set
  \begin{equation}\label{eq:D}
    D(x) := N_{1}(x)\cap \{y\in V\mid \deg(y)=d\}
   \end{equation}
 of vertices of degree $d$ that are at distance at most $1$ from $x$. 
(Note that if $\deg(x) < d$ then $x$ itself is not included into $D(x)$.)

Given a matching $M\subseteq E$, the star of $x\in K$ can be one of  three different types:
\begin{itemize}
\item \emph{Complete:}  if $D(x)\subseteq V(M)$, that is, all vertices of 
$D(x)$ are covered by the matching~$M$ (including the case
$D(x)=\emptyset$).
\item \emph{Heavy:} if $|D(x)\setminus V(M)|=1$, that is, exactly one vertex of $D(x)$ is not covered by~$M$.
\item \emph{Light:} if at least 2 vertices of $D(x)$ are not covered by $M$.
\end{itemize}

We define the \textit{truncated star}  $D'_M(x)$ of $x \in K$ to consist of all but one of the  
uncovered vertices of the star. The excluded vertex is arbitrary: we can take, for example, the one with the largest label in some $2$-sparse labeling of $\GG$ which is fixed throughout the whole proof,
\[ D'_M(x) : = (D(x) \setminus V(M)) \setminus \{\mbox{the largest remaining vertex, if any left}\}. \]

Note that the truncated star is empty if and only if the star is heavy or complete. 
We define the set of  \emph{unhappy}  vertices $U_M$ to contain all unmatched vertices of degree at least $d$ that are not in or adjacent to $K$ together with all vertices in truncated stars:
\begin{equation}\label{def:unhappy} 
 U_M :=   \Big(\,\{x\in V\setminus  N_{1}(K)\mid \deg(x) \geq d\} \setminus V(M)\, \Big)\, \cup\,  {\big(\cup_{x \in K} D'_M(x)\big)}.\end{equation}
 To motivate this definition, note that if we can find a Borel matching that covers $V(M)\cup U_M$ then 
Theorem~\ref{theorem:inductive_step} is proved.

When the current matching $M$ is understood, we may abbreviate $D'_M(x)$ and $U_M$ as respectively $D'(x)$ and~$U$.

\paragraph{Constructing the matching $M$:} 
First, we will construct a sequence of Borel matchings $M_0,M_1, M_2, \ldots\subseteq E$ such that the following properties hold:
 \begin{align}
 J&\ \subseteq\ V(M_i), & \mbox{for all $i\ge 0$,}\label{eq:1}\\ 
 \muN{V(M_i\bigtriangleup M_{i+1})}&\ \le\ (2i+2)\,\muN{U_{M_i}}, & \mbox{for all $i\ge 0$,}\label{eq:2}\\
 \sum_{j=0}^\infty (2j+2)\,\muN{U_{M_j}} &\ <\ \infty.\label{eq:3}
\end{align}

Once we have $M_i$'s as above, we can define
$M:=\cup_{j=0}^\infty \cap_{i=j}^\infty M_i$ to consist of those pairs that belong to all but finitely many matchings $M_i$. Clearly, $M\subseteq E$ is a Borel matching. 
Let us show that $M$ satisfies Theorem~\ref{theorem:inductive_step}. 

Since $V(M_i\bigtriangleup M_{i+1})$ is the set of vertices that experience
some change when we pass from $M_i$ to $M_{i+1}$, the last two
conditions imply by the Borel--Cantelli Lemma that the set $X$ of
vertices where the matchings $M_i$ do not
stabilize has measure zero. By (\ref{eq:1}), we conclude that
$J\setminus V(M)$ is a subset of $X$ and thus has measure zero. Also, the symmetric difference between $U_M$ and $\cup_{j=0}^\infty\cap_{i=j}^\infty U_{M_i}$ is contained within the null-set $N_{2}(X)$. Since $\sum_{i=0}^\infty \muN{U_{M_i}}$ converges by (\ref{eq:3}), each intersection 
$\cap_{i=j}^\infty U_{M_i}$ has measure zero. By the $\sigma$-additivity of $\mu$, we conclude that $U_M$ has measure zero too. 

Hence, if we remove all connectivity components intersecting
$U_M\cup X$, 
then $J\subseteq V(M)$ and all unmatched degree-$d$ vertices come
from stars $D(x)$, $x\in K$, at most one vertex per star.  Since the removed
set has measure zero and $J\cup K$ is $(r+2)$-sparse, 
all conclusions of  Theorem~\ref{theorem:inductive_step} hold.

\paragraph{Augmenting paths:} 
It remains to construct the sequence $(M_i)_{i\in \I N}$ satisfying the above three conditions. Inductively for $i\in\I N$,  we will construct $M_{i+1}$ from $M_i$ by flipping \emph{alternating paths}, that is, paths that start in an unmatched vertex and whose matched and unmatched edges follow in an alternating manner. 
\emph{Flipping} such a path means changing the matching along the path by swapping
the matched edges with the unmatched ones. Usually one would only flip paths that end in an unmatched vertex, thus strictly increasing the set of matched vertices. In our case, however, we may also have to flip paths whose last edge belongs to the matching. While such a flip retains the matching property, it does not increase its size, so extra care needs to be taken.

With all these preparations we can describe what kind of alternating paths we use to improve the 
current matching~$M$.

\begin{definition}\label{def:augmenting} An \textit{augmenting path} is an alternating path that starts in $U_M$ 
(that is, with an unhappy vertex) and   
  \begin{itemize}
  \item either has odd length and ends in any unmatched vertex,
 \item or has even length and ends in a vertex of degree less then $d$ or in a vertex of a complete star.
 \end{itemize} 
\end{definition}

When we have more than one matching involved, we may call a path as above \emph{$M$-augmenting}.

\begin{claim}\label{cl:U} If  $M'$ is obtained from $M$ by flipping an $M$-augmenting path $(p_0,\dots,p_k)$, then
$U_{M'}\subseteq U_M\setminus \{p_0\}$; in particular, the set of unhappy vertices strictly decreases.
\end{claim}

\begin{proof}
Clearly, $p_0$ belongs to $U_M$ but not to $U_{M'}$. So we only need to check 
that $U_{M'}\subseteq U_{M}$, that is, no new vertex can become unhappy because of a flip. This could only be possible by uncovering a vertex of degree at least $d$. The only way a flip can uncover a vertex is when the path has even length: in this case the endpoint is uncovered. But the endpoint of an even length augmenting path can only have degree at least $d$ if the augmenting path ends in a vertex of a complete star. After the flip, the star becomes heavy (but not light), so the uncovered vertex does not belong to $U_{M'}$. 
\qed\end{proof}

Now we are ready to describe formally the inductive construction of the matchings $M_i$.

\paragraph{Constructing the initial matching $M_0$:}
We start by constructing a matching $M_0$ that covers all vertices in $N_{r_1/4}(J)$ of degree
at least $d$. (This property will be needed later when we apply Theorem~\ref{shortalternating_theorem}
to verify \eqref{eq:3}.) 
 Since the neighborhoods $N_{r_1/4+1}(x)$ for $x\in J$ are disjoint (assuming $r_1\ge 2$), 
we can
choose $M_0$ inside each neighborhood independently, for example,
by taking the lexicographically smallest matching with respect to a fixed 
$r_1$-sparse labeling. This ensures that the obtained matching $M_0$ is Borel.

\hide{
\begin{lemma}\label{d+1covering_lemma} Let $d \geq 2$. Let $G=(V,E)$ be an infinite connected 
graph of maximum degree $d$ except one vertex $x$ of degree $d+1$.
Assume that $G$ is bipartite or has no stumps !!! we defined stumps only for infinite
conn graphs!!!. Then
$G$ has a matching that covers all vertices of degree at least $d$.
\end{lemma}
}

It remains to show that for for each $x\in J$ the finite subgraph $G'$ induced by
$V':=N_{r_1/4+1}(x)$ has a matching covering every vertex of  $Z := \{z \in N_{r_1/4}(x) : \deg(z) \geq d\}$.

First, let us prove the no-stumps case. Tutte's 1-Factor Theorem \cite{tutte:47} implies that it is enough to check that for every
set $S\subseteq V'$ the number of odd components of $G'-S$ that lie entirely 
inside $Z$ is at most $|S|$. (The reduction is as follows: add, if needed, an
isolated vertex to make $|V'|$ even, make all pairs in
$V'\setminus Z$ adjacent and look for a perfect matching in this new graph on~$V'$.) 
Suppose on the contrary that some  $S$
violates the above condition. Each odd
component $C$ 
of $G'-S$ that
lies inside $Z$ sends at least $d$ edges to $S$: this follows from the definition of $Z$ if $|C|=1$ and from the absence of stumps if $|C|\ge 2$. Thus
$|E(S,S^c)|\ge (|S|+1)d$. On the other hand, all vertices of $S$ have
degree at most $d$ except at most one vertex of degree $d+1$. 
Hence the total degree of $S$ is at most $|S|\,d+1$. Since $d\geq 2$, this is strictly less than $(|S|+1)d$, giving a contradiction. 

Now, let us do the case when there are no odd cycles. Split $Z = Z_1\cup Z_2$ into two parts according to the bipartition of the finite graph $G'$.
It is enough to show that there is a matching that covers $Z_1$ and 
one that covers $Z_2$. Indeed, the union of these two matchings consists of paths and even cycles;
moreover one endpoint of each path of even length (whose number of
vertices is odd) has to be outside of $Z$. By deleting such endpoints,
we can assume that $Z$ is covered by cycles and paths, each having an even
number of vertices. However, every such cycle and path admits a perfect 
matching, and the union of these matchings covers $Z = Z_1\cup Z_2$, as desired.

So suppose that there is no matching that covers, say, $Z_1$. By the 
K\H{o}nig-Hall theorem this means that there is a subset $S \subseteq Z_1$ such 
that the set of neighbors $T$ of $S$ has strictly smaller size than $S$. Each 
vertex in $S\subseteq Z$ has degree at least $d$, so the number of edges 
leaving $S$ is at least $d\,|S|$. However, the number of edges arriving in $T$ 
is at most $d\,|T| + 1\le d(|S|-1)+1$. Again, this contradicts
$d\ge 2$.

\hide{
\begin{lemma}\label{d+1covering_lemma} Let $r$ be an integer and $d \geq 2$. Let $G=(V,E)$ be an infinite connected graph of maximum degree
$d+1$. Let $J$ be the set of vertices of degree $d+1$ and let 
 $$
 Z := \{z \in  N_{r}(J) : \deg(z) \geq d\}.
 $$ 
 Assume that  $J$ is
$(2r+2)$-sparse and that $G$ is bipartite or has no stumps. Then
$G$ has a matching that covers all of $Z$.
\end{lemma}
}

\paragraph{Defining $M_{i+1}$ given $M_i$:} For $i\ge 0$, we define $M_{i+1}$ recursively so that $M_{i+1}$ admits no augmenting path of length at most $2i+1$. To get from $M_i$ to $M_{i+1}$ we keep flipping augmenting paths of length at most $2i+1$ as long as there are any of those left. This can be done in a Borel way analogously to how it was done in \cite{elek+lippner:10} as follows. 

First, fix a $(2i+3)$-sparse labeling $\ell:V\to [m]$. Let $\C L$ consist of all ordered sequences of labels of length at most $2i+2$. Take an infinite sequence $(\B v_j)_{j=0}^\infty$ of elements of $\C L$ such that each $\B v\in \C L$ appears infinitely often. Let $M_{i,0}:=M_i$. Then iterate  the following step over $j\in \I N$.
Let $\C P_j$ be the set of $M_{i,j}$-augmenting paths whose labeling is given
by $\B v_j$. It is easy to see that $\C P_j$ is a Borel set that consists of paths such that every two different
paths are at distance at least 3 from each other. Thus we can flip all paths in $\C P_j$ simultaneously,
obtaining a new matching $M_{i,j+1}$. (Note that at most one path can intersect any star $D(x)$; thus every path $p\in \C P_j$ remains
$M_{i,j}$-augmenting even when we flip an arbitrary set of paths in $\C P_j\setminus\{p\}$.) 
Clearly, the obtained matching $M_{i,j+1}$ is Borel. By Claim~\ref{cl:U}, all starting
points of the flipped paths cease to belong to $U_{M_{i,j}}$ while no new vertex can become unhappy.

Having performed the above iteration over all $j\in\I N$, we define
 $$
 M_{i+1}=\cup_{j_0=0}^\infty \cap_{j=j_0}^\infty M_{i,j}
 $$
 to consist of those edges that belong to all but finitely many of the matchings $M_{i,0},M_{i,1},\dots$~.
Clearly, $M_{i+1}$ is a Borel matching.

 Consider an arbitrary edge $e\in E$. Since the number of unhappy vertices within distance $2i$ from $e$ strictly decreases every time when a path containing $e$ is flipped, the matchings $M_{i,0},M_{i,1},\dots$ stabilize on $e$ from some moment.

Suppose that $M_{i+1}$ admits some augmenting path via vertices $p_0,\dots,p_k$ with $k\le 2i+1$. There was a moment $j_0$ after which every edge inside the path and inside $N_{3}(p_0)\cup N_{3}(p_k)$ stabilized. The restriction of the matching $M_{i+1}$ to these edges determines whether the path is $M_{i+1}$-augmenting.
But there are infinitely many values of $j$ when $\B v_j=(\ell(p_0),\dots,\ell(p_k))$
and this path should have been flipped for the first such $j\ge j_0$, a contradiction. Thus
$M_{i+1}$ has no augmenting path of length at most $2i+1$, as desired.

\paragraph{Checking Conditions (\ref{eq:1})--(\ref{eq:3}):} 
 The first condition trivially follows from our construction since
$V(M_0)\supseteq J$ while no flip can unmatch a vertex of degree
$d+1$.

When constructing $M_{i+1}$, we flipped some family ${\cal P}:=\cup_{j=0}^\infty {\cal P}_j$ of augmenting paths that start from $U_{M_i}$.
Let us show that the total measure of vertices involved in these paths is at most $(2i+2)\,\muN{U_{M_i}}$, thus proving~(\ref{eq:2}). 

Take any $j\in\I N$. For $0\le k\le 2i+1$, let $X_{j,k}$ consist
of vertices that are the $k$-th indexed vertex of some path in ${\cal P}_j$, that is,
for every path $(p_0,\dots,p_\ell)\in{\cal P}_j$ with $\ell\ge k$, we include $p_k$ into~$X_{j,k}$. Recall that the paths in ${\cal P}_j$ are vertex-disjoint. Thus, for odd (resp.\ even) $k\le 2i$, the matching
$M_{i,j}$ (resp.\ $M_{i,j+1}$) induces an injective function from $X_{j,k+1}$ into~$X_{j,k}$. Of course, if we
restrict the corresponding matching to those pairs that connect the Borel sets $X_{j,k+1}$ and $X_{j,k}$, then
we obtain a Borel subset of edges. We conclude from Proposition~\ref{pr:subgraphing}
that $\mu(X_{j,k+1})\le \mu(X_{j,k})$. Thus the measure of 
$V(M_{i,j+1}\bigtriangleup M_{i,j})=\cup_{k=0}^{2i+1} X_{j,k}$ is at most 
$(2i+2)\,\mu(X_{j,0})$. 

By Claim~\ref{cl:U}, each vertex of $U_{M_i}$ is the initial vertex of 
at most one path from $\cal P$. Also, the starting point of each path in $\C P_j$ belongs to $U_{M_{i,j}}\setminus U_{M_{i,j+1}}$, which is a subset of $U_{M_i}\setminus U_{M_{i+1}}$, again by Claim~\ref{cl:U}. Thus the sets $X_{j,0}$,
for $j\in\I N$, are pairwise disjoint and all lie in~$U_{M_i}\setminus U_{M_{i+1}}$. We conclude that the measure of      
 $$
 V(M_{i+1}\bigtriangleup M_i)\subseteq \cup_{j=0}^\infty V(M_{i,j+1}\bigtriangleup M_{i,j}) 
 $$
 is, by above, at most $(2i+2)\sum_{j=0}^\infty \mu(X_{j,0})\le 
 (2i+2)\,\mu(U_{M_i})$. This gives~(\ref{eq:2}),
as desired.

By definition,  $M_0$ covers all vertices of degree at least $d$ from $N_{r_1/4}(J)$. In particular, it
follows that $\dist(U_{M_0},J)>r_1/4$. Claim~\ref{cl:U} and the
fact that each edge is flipped finitely many times before we reach $M_i$
imply that $U_{M_0}\supseteq U_{M_{1}}\supseteq \dots \supseteq U_{M_i}$. 
Thus $\dist(U_{M_i},J)>r_1/4$. 
Since $M_i$ admits no augmenting paths of length at most $2i-1$, we can invoke the following
theorem with $n_0=2i-1$ provided that $r_1$ is chosen large enough (namely, if $r_1\ge r_2(d,r')$, the value returned by Theorem~\ref{shortalternating_theorem} on input $(d,r')$).

\begin{theorem}\label{shortalternating_theorem} 
For any $d \geq 2$ and $r'$ there are constants $c =
c(d,r')>0$ and $r_2 = r_2(d,r')$ such that the following holds. Let $\GG=(V,\C B,E,\mu)$  be a graphing such that $\Delta(\GG)\le d+1$, all components are infinite and the set $J$ of vertices of degree $d+1$ is $r_2$-sparse. 
Let $K\subseteq V\setminus N_1(J)$ be a Borel  
set of vertices that is $r'$-dense in $\GG$. Let $M\subseteq E$ be a Borel matching that admits no augmenting path
of length at most $n_0$. Let $U_M$ be  the set of unhappy vertices with respect to $M$ and $K$ as defined in \eqref{def:unhappy}.
If  $\dist(U_M,J)>r_2/4$ and $\GG$ has no odd cycles or has no stumps, then $\muN{U_M}\le (1+c)^{-n_0}/c$.
\end{theorem}

The theorem gives that $\muN{U_{M_i}}\le (1+c)^{1-2i}/c$, where $c=c(d,r')>0$ does not depend on~$i$.
This means that $\muN{U_{M_i}}$ decreases exponentially fast with $i$, implying~(\ref{eq:3}).
 
Thus we have proved Theorem~\ref{theorem:inductive_step} (by reducing it to Theorem~\ref{shortalternating_theorem}).\qed\medskip

\section{Short alternating paths via expansion}\label{section:shortaugmenting}

In this section we prove Theorem~\ref{shortalternating_theorem}, which will complete the proof of our main result. Roughly speaking, Theorem~\ref{shortalternating_theorem} says that if $M$ is a matching in $\GG$ such that there are no augmenting paths of length at most $n_0$ in the sense of Definition~\ref{def:augmenting}, then the set $U_M$ of unhappy vertices is exponentially small in~$n_0$.
The proof method is adapted from~\cite{csoka+lippner} but is also considerably simplified since 
we have a dense set around which unmatched vertices are allowed. Essentially all proof ideas that went
into Theorem~\ref{shortalternating_theorem} are already contained in~\cite{csoka+lippner}. Nevertheless we include here all the details to keep this paper self-contained. The main idea is that if 
the special set $K$ is dense then subsets of $K^c$ expand by Lemma~\ref{lemma:expansion}; thus the set of vertices that can be reached by alternating paths
of length at most $i$ grows exponentially with $i$. This is fairly straightforward to show in the case when there are no odd cycles. This proof is presented first (in Section~\ref{NoOddCycles}),  to make it easier for the reader to understand the ideas that are common to both cases. The general case, however, needs some further, rather involved arguments. This is due to the fact that an alternating walk from $x$ to $y$ cannot always be trimmed to an alternating path from $x$ to $y$ when odd cycles are allowed.

\subsection{Alternating breadth-first search}\label{alternating}

Given $d$ and $r'$, define $c_0:=d^{-r'}$. Let $r_2 = r_2(d,r')$ be sufficiently large and then take small $c>0$.
Let $\GG$, $K$ and $M$ be as in Theorem~\ref{shortalternating_theorem}. By assuming that e.g.\ $(1+c)^{-4}/c\ge 1$,
it enough to establish the claim for $n_0\ge 5$ only.

We consider alternating paths starting from~$U=U_M$.
Let $X_n$ denote the set of vertices that are accessible from $U$ via an alternating path of length at most~$2n$. Our first goal is to show that subsets of $X_n$ have large boundary if $2n\le n_0$.

For $n\ge 1$, let $\tilde{H}_n$ denote the vertices that are endpoints of odd alternating paths of length at most $2n-1$ and $\tilde{T}_n$ those that are endpoints of even alternating paths of length at least 2 and at most $2n$. Then $X_n = U\cup \tilde{H}_n \cup \tilde{T}_n$. Finally, define $O_n := V \setminus X_n$.

\begin{propo}\label{prop:prop1} If $1\le n \le n_0/2$, then the following claims hold.
\begin{enumerate}
\item The points in $\tilde{H}_n$ are covered by $M$.
\item The edges of the matching $M$ give a bijection between $\tilde{H}_n$ and $\tilde{T}_n$. 
(In particular, $\muN{\tilde{H}_n} = \muN{\tilde{T}_n}$ and $\tilde{H}_n\cup \tilde{T}_n \subseteq V(M)$ is disjoint from 
$U$.) 
\item $\tilde{T}_n$ cannot contain vertices of degree less than $d$.
\item $K \cap U = \emptyset$.
 \item Every $x\in K$ with $D(x)\subseteq X_n\setminus U$ belongs to $O_n$.
\end{enumerate}
\end{propo}

\begin{proof}
Part 1 is clear: if a vertex in $\tilde{H}_n$ would not be matched then it would give rise to an augmenting path of length at most $2n-1$. 

Part~2 follows from Part~1 and the definition of an alternating path. (Note that $\muN{\tilde{H}_n} = \muN{\tilde{T}_n}$ because $(V,\C B,M,\mu)$ 
is a graphing by Proposition~\ref{pr:subgraphing}.) 

Part~3 is immediate from the definition of an augmenting path because a vertex of small degree in $\tilde{T}_n$ 
gives an augmenting path of length at most $2n$.

To see Part~4, assume that there is $x \in K \cap U$.  Then all neighbors of $x$ are matched, for otherwise we get an augmenting path of length 1.  Thus $D(x)$ is a complete or heavy star. But then $x$ cannot belong to $U$ by definition.

Let us prove Part~5. The assumption $D(x)\subseteq X_n\setminus U$ implies by 
Part~2 that the star of $x$ is complete. 
If $x\in \tilde{T}_n$, then there is an even alternating path  of length at most $2n$ ending in $x$. But then this path is also 
augmenting, a contradiction. Suppose next that $x\in \tilde{H}_n$. Let $y$ be 
the last vertex before $x$ on an alternating 
path $p$ of length $2m-1\le 2n-1$ from $U$ to $x$. We have that $m\ge 2$ for
otherwise $y\in U$ has degree $d$  and so $D(x)\cap 
U\ni y$ is
non-empty, contrary to our assumption. Part~3 implies that $y\in \tilde{T}_{m-1}$ has degree at least $d$. In fact, $y\in N_1(x)$ has degree exactly $d$,
since we assumed that $J\cap N_1(K)=\emptyset$.
It follows that $y$ belongs to $D(x)$ and the path $p' := p \setminus x$ is  augmenting because $D(x)$ is complete, a contradiction which shows that $x\not\in \tilde{H}_n$. Also,
the vertex $x\in K$ cannot belong to $U$ by Part~4. We conclude that $x$ lies in $(\tilde{T}_n\cup \tilde{H}_n\cup U)^c=O_n$, as required.
\qed\end{proof}

\begin{lemma}\label{lemma:denseset} If $1\le n\le n_0/2$, then $O_n$  is $(r'+1)$-dense. 
\end{lemma}

\begin{proof} Since $K$ is $r'$-dense by the assumption
of Theorem~\ref{shortalternating_theorem}, it is enough to show that
$O_n\cap N_1(x)\not=\emptyset$ for every $x\in K$. 

Take any $x\in K$. Assume that $D(x)$ is not a subset of $X_n\setminus U$ for
otherwise the vertex $x$ itself belongs to 
$O_n$ by Proposition~\ref{prop:prop1}.5. If $D(x)$ intersects
$O_n$ then we are done, so we can additionally assume that $D(x)\cap 
U\not=\emptyset$.
This means that the star of $x$ is light. Let $y$ be the unique element of 
$D(x) \setminus D'(x)$. By the definition of $U$, we have that $y \not \in U$. 
Since $y$ is unmatched, it cannot belong to $\tilde{T}_n\cup \tilde{H}_n$ by 
Proposition~\ref{prop:prop1}.2.
We conclude that $N_1(x)\cap O_n\ni y$ is non-empty, as required.
\qed\end{proof}

Thus every subset of $X_n$ has large boundary by the following result.

\begin{lemma}\label{lemma:expansion} If $Q\subseteq V$ is $(r+1)$-dense in a graphing $\GG=(V,{\cal B},E,\mu)$ of degree bound $d+1$, then the measure of edges leaving any Borel subset $W \subseteq Q^c$ is at least $d^{-r} \muN{W}$.
\end{lemma}

\begin{proof} 
 For every $w\in W$ there is a path of length at most $r+1$ that
goes from $w$ to $Q$. Since $Q\subseteq W^c$, this path contains at least
one edge that connects $W$ to $W^c$. On the other hand, any
edge can arise this way for at most $d^{r}$ different vertices $w\in W$.
The required bound 
can now be derived from the so-called 
Mass Transport Principle (see e.g.\ \cite[Proposition 18.49]{lovasz:lngl}) which, roughly speaking, states that direct double-counting
(in)equalities from finite graphs also apply to the measures of vertex or edge sets in a graphing. 

For reader's convenience, let us sketch a more direct proof  of the
last step, so that we rely
only on the results stated in Section~\ref{basic}. Let $\phi_1,\dots,\phi_k$
generate the graphing as in Definition~\ref{df:graphing}. For every
sequence $f=(f_1,\dots,f_\ell)$ over the alphabet $\{\phi_1^{\pm1},\dots,\phi_k^{\pm1}\}$ with $1\le \ell\le r+1$, let $W_f$ consist of those $w_0\in W$ such that
$w_i:=f_i(f_{i-1}(\dots f_2(f_1(w_0))\dots))$ is defined for each  $i\in [\ell]$, $(w_0,\dots,w_\ell)$ is a path in $\GG$, and $w_\ell$ is the unique vertex
of this path that belongs to $W^c$; also, let $\pi_f$ map $w_0\in W_f$ to $(w_{\ell-1},w_\ell)\in E(W,W^c)$. It is
possible that different sequences $f_1$ and $f_2$ give the same path
$(w_0,\dots,w_\ell)$ for some starting point $w_0\in W_{f_1}\cap W_{f_2}$. 
We shrink the sets $W_f$ by picking for each obtained path exactly one sequence that gives it, 
e.g.\ the lexicographically
smallest one. This ensures that each edge in $E(W,W^c)$
belongs to the image of at most $d^{r}$ maps~$\pi_f$ while the sets $W_f$
are still Borel and cover the whole of~$W$. Now, the desired
bound follows, since each $\pi_f$ is a measure-preserving bijection from
$(W_f,\mu)$ to its image in $(E(W,W^c),\muC)$.
\qed\end{proof}

\subsection{Proof for graphings without odd cycles}\label{NoOddCycles}

We have all the ingredients to finish the proof of Theorem~\ref{shortalternating_theorem} in the case 
when $\GG$ has no odd cycles.  In this section, let $n\in \I N$ be arbitrary with $1\le n\le (n_0-1)/4$. 

First, let us show that $\tilde{H}_n$ and $\tilde{T}_n$ are disjoint. Assume for a 
contradiction that $x \in \tilde{H}_n \cap \tilde{T}_n$. Then there are two vertices $u_1, u_2\in U$ and an odd alternating path from $u_1$ to $x$ and an even alternating path from $u_2$ to $x$. The concatenation of these two paths (with the second path being reversed) is an odd alternating \textbf{walk} from $u_1$ to $u_2$. Since $\GG$ has no odd cycles,
we have that $u_1\not=u_2$. Also, we conclude that there is an odd alternating \textbf{path} from $u_1$ to $u_2$, since a shortest alternating walk in a bipartite graph is necessarily an alternating path. This 
path has length at most $4n-1\leq n_0$ and is augmenting, a contradiction.

A similar argument shows that there can be no edge within $\tilde{T}_n\cup U$ 
for otherwise we find an augmenting path of length at most $4n+1\le n_0$.

Any vertex outside of $X_n$ that is adjacent to $\tilde{T}_n$ will belong to $\tilde{H}_{n+1} \subseteq X_{n+1}$. We want to show that there are many such vertices, so we derive a lower bound on the measure of edges leaving $\tilde{T}_n$. 
By Proposition~\ref{prop:prop1}.2, we know that $\muN{\tilde{T}_n}=\muN{\tilde{H}_n}$. Also, every vertex of $\tilde{T}_n$ has degree at least $d$ while vertices of degree $d+1$ are $r_2$-sparse. Since the set
$\tilde{T}_{n}\cup U$ is independent, $U$ sends no edges to $O_n$ and $r_2$ is 
large,
we would expect that at least around half of the edges between $X_n$ and $O_n$ originate from $\tilde{T}_n$. The following inequalities make this intuition rigorous. For notational convenience, let
 \begin{equation}\label{eq:muCe}
 \muCNe XY:=\muCN{E(X,Y)},\qquad \mbox{for $X,Y\subseteq V$},
 \end{equation}
 denote the measure of edges between $X$ and $Y$. We have
 \begin{eqnarray*} \muCNe{\tilde{T}_n \cup U}{O_n} & \geq&
 d\,\muN{\tilde{T}_n} - \muCNe{\tilde{H}_n}{\tilde{T}_n \cup U} 
 \\ 
 &\geq& 
 d\,\muN{\tilde{H}_n} - \muCNe{\tilde{H}_n}{X_n}\ \geq \ 
 \muCNe{\tilde{H}_n}{O_n} - \muN{\tilde{H}_n \cap J} \\ 
 &=& 
\muCNe{X_n}{O_n} - \muCNe{\tilde{T}_n \cup U}{O_n} 
 - \muN{\tilde{H}_n \cap J}.
 \end{eqnarray*}
 Hence
\begin{equation}\label{eq:TU-O}
 \muCNe{\tilde{T}_n \cup U}{O_n} \geq \frac{1}{2}\big(\,\muCNe{X_n}{O_n} 
 - \muN{\tilde{H}_n \cap J}\,\big).
 \end{equation}
Recall that $c_0=d^{-r'}$.
By Lemmas~\ref{lemma:denseset} and~\ref{lemma:expansion},
the measure of edges leaving $X_n$ is at least $c_0\, \muN{X_n}$.

Take, for each $x \in \tilde{H}_n \cap J$, a shortest alternating path from $U$ to $x$. Its length is
at least $r_2/4$ because $\dist(U,J)>r_2/4$
by our assumption.  Moreover, since $J$ is $r_2$-sparse, the final $r_2/4$ edges of this path are unique to $x$: for different vertices of $\tilde{H}_n \cap J$ these segments are disjoint (and, obviously, these segments belong to $X_n$). Since these paths can be chosen in a Borel way, we conclude that
\begin{equation}\label{eq:HCapJ} 
 \muN{\tilde{H}_n\cap J} \leq \frac{4}{r_2}\, \muN{X_n}.
 \end{equation}

Assuming that $4/r_2  < c_0 /2$, we have by~\eqref{eq:TU-O} and~\eqref{eq:HCapJ} that 
\[
 (d+1)\,\muN{X_{n+1}\setminus X_n}\ge \muCNe{\tilde{T}_n \cup U}{O_n} \geq \frac{c_0}{2}\,\muN{X_n} - \frac{c_0}{4}\,\muN{X_n} =  \frac{c_0}{4}\,\muN{X_n}.
 \] We get by induction on $n$ that
 \[1
 \geq  \muN{X_{n+1}} \geq \left(1+\frac{c_0}{4(d+1)}\right)\, \muN{X_n} \geq 
 \left(1 + \frac{c_0}{4(d+1)} \right)^{n+1} \muN{U}.\]
  In particular, by taking $n =\lfloor(n_0-1)/4\rfloor$ we conclude that $\muN{U}\le (1+c)^{-n_0}/c$, as desired.

\subsection{Sketch of the proof in the general case}\label{sketch}

We continue using the notation introduced in Section~\ref{alternating}
but we need a more refined analysis of different types of vertices in $X_n$ than the one in Section~\ref{NoOddCycles}. Since odd cycles are allowed, the sets $\tilde{H}_n$ and $\tilde{T}_n$ need not be disjoint. It will be convenient to introduce the following notation:
\begin{eqnarray*} 
 H_n &:=& \tilde{H}_n \setminus \tilde{T}_n,\\
 T_n &:=& \tilde{T}_n \setminus \tilde{H}_n,\\
 B_n &:=& \tilde{H}_n \cap \tilde{T}_n.
\end{eqnarray*}
 Here, $H$ stands for ``head'', $T$ stands for ``tail'', and $B$ stands for ``both''. These sets satisfy the following simple properties in addition to those already stated in Proposition~\ref{prop:prop1}.

\begin{propo}\label{prop:prop2} If $1\le n\le n_0/2$, then the following properties hold.
\begin{enumerate} 
\item $X_n$ is the disjoint union of $U$, $T_n$, $H_n$ and $B_n$.
\item $B_1 \subseteq \dots \subseteq B_n$.
\item $M$ gives a perfect matching between $T_n$ and $H_n$, and also within $B_n$. In particular,
$\muN{H_n} = \muN{T_n}$.\qed
\end{enumerate}
\end{propo}

Now we are ready to sketch the proof of Theorem~\ref{shortalternating_theorem}, 
pointing out the main ideas without introducing all technicalities.
We encourage the reader to study the whole outline before reading the proof and to refer back to it whenever necessary. Without understanding the basic outline, some later definitions may seem unmotivated. 

\begin{enumerate}
\item Assuming there are no short augmenting paths, we would like to show that $\muN{X_n}$ 
grows exponentially with $n$.

\item By Lemmas~\ref{lemma:denseset} and~\ref{lemma:expansion}, the set $X_n$ expands.
If there are plenty of edges leaving $X_n$ from $T_n$ or $B_n$, then the other ends of these edges will be part of $X_{n+1}$, fueling the desired growth. If this is not the case, then there has to be many tail-tail or tail-both edges for the same reasons as in Section~\ref{NoOddCycles}: $\muN{H_n} = \muN{T_n}$, every vertex of $T_n$ has degree at least $d$ while
only a small fraction of vertices of $H_n$ can have degree $d+1$.

\item A tail vertex that has another tail- or both-type neighbor will normally become a both-type vertex in the next step. In this case, even if $X_n$ does not grow, the set $B_n$ grows within $X_n$, still maintaining the desired expansion.

\item The problem is that certain tail-vertices will not become both-type, even though they possess a both-type neighbor. These will be called \textit{\tough}. The bulk of the proof is about bounding their number. The key idea here is that we can associate to each \tough\ vertex $x$ a subset $F_n(x)$ of $B_n$ called the \textit{family} of $x$. 

\item As we will see in Lemma~\ref{disjointfamilies}, families associated to different vertices are pairwise disjoint. Thus there cannot be too many \tough\ vertices with large families. On the other hand,
Claim~\ref{expanding_claim} shows that if a vertex stays \tough\ for an extended amount of time, then its family has to grow. These two observations will be the basis for showing
that $B_n$ grows within $X_n$, thus indirectly contributing to the growth of $X_n$. 
\end{enumerate}

The proof is organized as follows.  We define \tough\ vertices and their families  in Section~\ref{combsec} 
where  their basic properties are stated and proved. Theorem~\ref{shortalternating_theorem} is  proved in Section~\ref{largeoutside_section} by introducing a function $I(n)$ that exponentially 
grows with $n$ for $n\le (n_0-2)/2$, is bounded by a constant and satisfies $I(0)=\muN{U}$. This will
give the desired upper bound on the measure of~$U$.

\subsection{Combinatorics of alternating paths}\label{combsec}

In this section we will be mainly concerned about how edges within $T_n \cup U$ and between $B_n$ and $T_n \cup U$ contribute to the growth of $B_n$. We implicitly assume  in all following claims that $1\le n\le (n_0-1)/2$. In particular,
there are no augmenting paths of length at most $2n+1$.

\begin{lemma}\label{TTedge_lemma} If $x,y \in T_n \cup U$ and $\{x,y\} \in E$, then $x \in B_{n+1}$ or $y \in B_{n+1}$.
\end{lemma}

\begin{proof} It is sufficient to prove that  $x$ or $y$ is in $\tilde{H}_{n+1}$. Let $p$ 
and $q$ be shortest alternating paths that witness $x\in \tilde{T}_n$ and $y \in \tilde{T}_n$ respectively. We may assume without loss of generality that $\len{p} \leq \len{q}$. (Recall that e.g.\ $\len{p}$ denotes
the number of edges in the path $p$.) The vertex $y$ cannot lie on $p$: otherwise either there would be a shorter alternating path witnessing $y \in T_n$, or we would have $y \in \tilde{H}_n$ and not in $T_n \cup U$. Hence, by adding the edge $\{x,y\}$ to $p$ we obtain an alternating path of length at most $2n+1$ that witnesses $y \in \tilde{H}_{n+1}$.
 \qed\end{proof}

Edges running between $T_n \cup U$ and $B_n$ are more complicated to handle. If $b \in B_n$ and $t \in T_n \cup U$ are adjacent, but all paths witnessing $b \in \tilde{T}_n$ run through $t$, then we cannot simply exhibit that $t \in \tilde{H}_{n+1}$ by adding the  edge $\{b,t\}$ to the end of such a path since it would become self-intersecting. The following definition captures this behavior.

\begin{definition}\label{tough_def}~ 
\begin{itemize} \item A vertex $x \in T_n \cup U$ is \emph{\tough\ (at time $n$)} if it is adjacent to one or more vertices in $B_n$, but $x \not \in \tilde{H}_{n+1}$. 
\item An edge $\{x,y\} \in E$ is \emph{\tough\ (at time $n$)} if $x \in T_n \cup U$, $y \in B_n$ and $x$ is a \tough\ vertex.
\end{itemize}
 Let $\TT_n\subseteq T_n\cup U$ denote the set of vertices that are \tough\ at time $n$.
\end{definition}

 We would like to bound the number of \tough\ vertices. In order to do so, we will associate certain subsets of $X_n$ to each \tough\ vertex in a way that subsets belonging to different \tough\ vertices do not intersect. Then we will show that these subsets become large quickly.

\begin{remark} We think of $n$ as the time variable, and all the sets evolve as $n$ changes. Usually $n$ will denote the ``current" moment in this process. In the following definitions of age, descendant, and family, there will be a hidden dependence on $n$. When talking about the age or the family of a vertex, we always implicitly understand that it is taken at the current moment. 
\end{remark}

\begin{definition}\label{age_def}
The \emph{age} of a vertex $x \in \TT_n$ is $a(x):=n$ if $x\in U$ and $a(x) := n - \min\{k: x \in T_k\}$ otherwise.
\end{definition}

\newcommand{\weaklyalternating}{weakly-alternating} 
Let us call a path \emph{\weaklyalternating} if it starts with an unmatched
edge and its edges alternate between $E\setminus M$ and $M$. This
is the same as the definition of an alternating path except we do not
require that the first vertex is unmatched.

\begin{definition}\label{descendent_def} Fix a vertex $x\in \TT_n$. A set $D \subseteq X_n\setminus\{x\}$ has the \emph{descendant property} with respect to $x$ if the following is true. For every $y \in D$ there are two \weaklyalternating\ paths $p$ and $q$ starting in $x$ and ending in $y$, such that 
\begin{itemize}
\item 
 $p$ is odd and $q$ is even,
\item $p, q \subseteq D \cup \{x\}$,
\item $\len{p} + \len{q} \leq 2a(x)+1$.
\end{itemize}
\end{definition}

Note that the paths $p$ and $q$ in the above definition may intersect outside of $\{x,y\}$. Clearly, the sets satisfying the descendant property with respect to $x$ are closed under union.

\begin{definition}\label{family_def} The \emph{family $F_n(x)$} of a vertex $x\in \TT_n$ at time $n$ is 
the largest subset of $X_n\setminus\{x\}$ that satisfies the descendant property with respect to~$x$.
(In other words, $F_n(x)$ is the union of all sets that satisfy the descendant property.) 
\end{definition}

\begin{claim}\label{claim:tough_edge} If $x \in \TT_n$ and $\{x,y\}$ is a \tough\ edge then $y$ is in the family of $x$. In particular, every \tough\ vertex has a non-empty family. 
\end{claim}

\begin{proof} Let $p$ be a path that witnesses $y \in \tilde{T}_n$. Now if $p$ appended by the edge $\{y,x\}$ would be a path then it would witness $x \in \tilde{H}_{n+1}$. Since this is not the case,  $x$ has to lie on $p$. Suppose $x = p_{2l}$ and $y = p_{2k}$. Let $D$ denote the set of vertices that the path $p$ visits after leaving~$x$. For any point $z \in D$ there are two \weaklyalternating\ paths from $x$ to $z$. One is given by following 
$p$ from $x$ to $z$ and the other by taking the edge $\{x,y\}$ and then walking backwards on $p$.  The total length of these two paths is $2k-2l +1$. Since the age of $x$ by Definition~\ref{age_def} is at least $n-l\ge k-l$ we see that $2k-2l+1 \leq 2a(x)+1$. Hence these two paths satisfy all conditions of Definition~\ref{descendent_def} and $D$ has the descendant property with respect to $x$. 
We conclude by Definition~\ref{family_def} that $y\in F_n(x)$. 
 \qed\end{proof}

\begin{claim} \label{cl:subsetB} The family of any \tough\ vertex is a subset of $B_n$. 
\end{claim}

\begin{proof} Let $x \in \TT_n$ be a \tough\ vertex and let $s$ be a shortest alternating path witnessing $x \in T_{n} \cup U$. Let us denote $k:=\len{s}/2$.

Let us show that the family of $x$ is disjoint from the path $s$. Suppose that this fails. It is clear that any family consists of pairs of matched vertices. Since $s$ is an alternating path, there is $i$ such that the edge
$\{s_{2i-1}, s_{2i}\}$ belongs to $M$ and lies inside $F_n(x)$. Let $i$ be the smallest such index. 
Then, by Definition~\ref{descendent_def},  there is an odd \weaklyalternating\ path ${p}$ from $x$ to $s_{2i}$   such that $p$ runs within the family and its length is at most $2a(x)+1 \leq 2n-2k+1$. 
Since $i$ was the smallest such index, the path ${p}$ is disjoint from $s_0, s_1, \dots s_{2i-1}$. 
Thus by appending $s_0, s_1,\dots,s_{2i}$ by the reverse of ${p}$ we get an alternating path from 
$U$ to $x$ ending in an unmatched edge, whose length is at most $2i+2n-2k+1 \leq 2n+1$. This path witnesses $x \in \tilde{H}_{n+1}$, contradicting that $x$ is  \tough.  

Now, for any point $y$ in the family we can take the two paths $p$ and $q$  from $x$ to $y$ as in Definition~\ref{descendent_def}. By the age requirement in Definition~\ref{descendent_def}, we get that $\len{p} + \len{q} \leq 2a(x) + 1 = 2n - 2k +1$.  Hence $\len{s} + \len{p} + \len{q} \leq 2n+1$ and thus $\len{s}+\len{p} \leq 2n-1$ and $\len{s} + \len{q} \leq 2n$. Since $p$ and $q$ run within the family (which is disjoint from $s$ as we have just established), we can append $s$ with ${p}$ and ${q}$ respectively to get alternating paths witnessing $y \in \tilde{H}_n$ and $y \in \tilde{T}_n$ respectively. Thus $y\in B_n$, as required.
 \qed\end{proof}

Next we will prove that any vertex can belong to at most one family. We start with a simple lemma about concatenating alternating paths.

\begin{lemma}\label{lm:xyz} Let ${p}$ be an even alternating path from $x$ to $y$. Let ${q}$ be an odd \weaklyalternating\ path from $y$ to $z$.
Then there is an odd alternating path from $x$ to either $y$ or $z$ whose length is at most $\len{p} + \len{q}$. 
\end{lemma}

\begin{proof} Note that $p$ ends with a matched edge and $q$ starts with an unmatched edge. If the concatenation of ${p}$ and ${q}$ is a path, then we are done. Otherwise let $i$ be the smallest index such that $p_i \in {q}$. Let $p_i = q_j$. Then $p_0,p_1,\dots, p_i = q_j, q_{j+1}, \dots, z$ is a path from $x$ to $z$ and $p_0,p_1,\dots,p_i=q_j, q_{j-1},\dots q_0$ is a path from $x$ to $y$. Both have length at most $\len{p}+\len{q}$, both of them end with non-matched edges and one of them is clearly alternating. 
 \qed\end{proof}

\begin{claim}\label{disjointfamilies} Two families cannot intersect.
\end{claim}

\begin{proof} Let $x, y \in \TT_n$ be two \tough\ vertices. Suppose that their families $F:=F_n(x)$ and $G:=F_n(y)$ do intersect. Let ${p}$ and ${q}$ be shortest alternating paths witnessing $x, y \in T_n \cup U$. Let us choose a shortest path among all \weaklyalternating\ paths from $x$ to $F \cap G$ that run within~$F$. Let this path be ${p'}$ and its endpoint be $x' \in F \cap G$. Do the same with $y$ to get a path ${q'}$ from $y$ to $y' \in F \cap G$ lying within $G$. By symmetry we may assume that $\len{p} + \len{p'} \leq \len{q} + \len{q'}$. 

By the choice of ${p'}$ we see that the only point on ${p'}$ that is in $G$ is its endpoint~$x'$. 
From $x'$ there are two paths, ${s}$ and ${t}$,  leading to $y$ within $G$ by Definition~\ref{descendent_def} one of which, say ${s}$, can be appended to ${p'}$ to get a \weaklyalternating\ path from $x$ to $y$. 

Now we are in the position to apply the previous lemma. The path ${p}$ leads from $p_0$ to $x$ and ends with a matching edge. The path ${p'} \cup {s}$ leads from $x$ to $y$ and starts and ends with non-matching edges. Thus by the lemma, there is an alternating path from $p_0$ to either $x$ or $y$ which ends with a non-matching edge. The length of this alternating path is at most $\len{p} + \len{p'} + \len{s}$. 
But by the choice of ${p'}$, the choice of ${q'}$, and by the age requirement in Definition~\ref{descendent_def} we have 
\begin{eqnarray*} 
 \len{p} + \len{p'} + \len{s} &\leq& \len{q} + \len{q'} + \len{s}\\
 &\leq& \len{q} + \len{t} + \len{s} \ \leq\ \len{q} + 2a(y) + 1\ =\ 2n+1.
 \end{eqnarray*}
Thus the alternating path from $p_0$ to $x$ or $y$ that we have found has length at most $2n+1$ and it witnesses $x \in \tilde{H}_{n+1}$ or $y \in \tilde{H}_{n+1}$. But neither is possible since both $x$ and $y$ are \tough, which is a contradiction.
 \qed\end{proof}

\begin{claim}\label{cl:onlyonetough} There is exactly one \tough\ vertex adjacent to any family.
\end{claim}

\begin{proof} Let $x,y \in \TT_n$ and $z \in F_n(x)$. Suppose there is an edge between $y$ and $z$.
The vertex $z$ is in $B_n$ by Claim~\ref{cl:subsetB}.  Hence $\{y,z\}$ is a \tough\ edge and $z$ is in the family of $y$ by Claim~\ref{claim:tough_edge}. But then the two families would not be disjoint, which is a contradiction to Claim~\ref{disjointfamilies}.
 \qed\end{proof}

Define $\cc:=2d/c_0$. (Recall that $c_0=d^{-r'}$.)  Roughly speaking, the following claim states that if a vertex remains \tough\ for an extended period of time, then its family, if it is not large already, consumes its neighbors.

\begin{claim}\label{expanding_claim}  Suppose that
$2(n+\cc)+1\le n_0$, $x\in\TT_n$,
 $|F_n(x)| < \cc$, $v \in F_n(x)$ and there is an edge $\{v,w\}$ such that $w \in B_n \setminus F_n(x)$.
If  $x \in\TT_{n+\cc}$, then $w \in F_{n+\cc}(x)$.
\end{claim}

\begin{proof} By definition, $x \not \in \tilde{H}_{n+\cc+1}$ as  $x$ is still \tough\ at the moment $n+\cc$.

First suppose that there is an even alternating path $p$ with $\len{p} \leq 2n$ that ends in $w$ and does not pass through $x$. Let $w' \in p$ be the first even vertex on the path that is adjacent to some vertex $v' \in F_n(x)$. Then the initial segment of $p$ up until $w'$ has to be disjoint from $F_n(x)$. By definition, in $F_n(x)$ there has to be a \weaklyalternating\ path from $x$ to $v'$ that ends in a matched edge. Extending this path through $w'$ and then the initial segment of $p$, we get an alternating path from $U$ to $x$. Its length is obviously at most $\len{p} + \cc$, hence $x \in \tilde{H}_{n+(\cc+1)/2}$ and consequently in $\tilde{H}_{n+\cc}$, which is a contradiction.

 This means that every even alternating path from $U$ to $w$ of length at most  $2n$ has to pass through $x$. 
Let $p$ be a shortest such path. Let $v'$ be the last vertex of $p$ that is in $F_n(x) \cup \{x\}$. The vertex $v'$ divides $p$ into two segments, $p_1$ going from $U$ to $v'$ and $p_2$ from $v'$
to $w$. Then 
 $$
 \len{p_2} = \len{p} -\len{p_1} \le 2n - 2 \min \{ k\ge 0 : x \in T_k \cup U\}
= 2a(x).
 $$ 

We claim that $p_2$ becomes part of the family at time $n+\cc$.
For any vertex $y \in p_2$ we can either go from $x$ to $y$ along $p$, or go from $x$ to $v$ in the even number of steps, then to $w$ and continue backwards on $p_2$ to $y$. The total length of these two paths is at most  $\cc + \len{p_2} + 1 + \cc \leq 2(a(x)+\cc) + 1$.  Since at moment $n + \cc$ the age of $x$ will be exactly $a(x)+\cc$, the set $F_n(x) \cup p_2$ will satisfy the descendant property, so this whole set, including $w$, will be a part of $F_{n+\cc}(x)$.
\qed\end{proof}

\begin{definition}\label{expanding_def} We will say that at moment $n$ the family of the vertex $x \in \TT_n$ is \textit{expanding} if  there is an edge $\{v,w\}$ such that $v \in F_n(x)$ and $w \in B_n \setminus F_n(x)$.
For any $x \in V$, let $e_n(x)$ be the number of moments $m < n$ such that $x\in \TT_m$, $0 < |F_m(x)| < \cc$ and at moment $m$ the family of $x$ was expanding. 
\end{definition}

\begin{claim}\label{cl:bounded_en_corol} For any $x\in V$ and $n\le (n_0-1)/2$, we have $e_n(x) \leq \cc^2$. 
\end{claim}

\begin{proof} By Claim~\ref{expanding_claim} we know that the number of moments in which an expanding family has a fixed size $k < \cc$ is at most $\cc$. This is because, within $\cc$ steps after the first such moment,  the family either ceases to exist (as the vertex $x$ is not \tough\ anymore) or strictly grows. Thus for each possible size $k$ there are at most $\cc$ moments of expansion, and thus there are at most $\cc^2$ such moments in all.
\qed\end{proof}

\subsection{Invariants of growth}\label{largeoutside_section}

Now we are ready to finish the proof of Theorem~\ref{shortalternating_theorem} for graphings without stumps. Let all the previous definitions  and results apply (except those from Section~\ref{NoOddCycles}, obviously). We restrict ourselves to those $n$ that are at most $(n_0-2)/2$.

As we have seen, a fairly short computation was enough to show that $\muN{X_n}$ grows exponentially when we had no odd cycles. 
In the general case,  we need to use a more complicated invariant than $\muN{X_n}$ as a measure of growth. Namely, we consider
\[
I(n) := \muN{X_n} + \muN{B_n} + \frac{1}{2}\int_{X_n} e_n(x) \dd x.
\]

Recall that $\TT_n\subseteq T_n \cup U$ denotes the set of \tough\ vertices. Let $\TM_n:=(T_n\cup U)\setminus \TT_n$ be the set of non-\tough\ vertices within $T_n \cup U$. The \tough\ vertices in $\TT_n$
are further classified according to their families. Namely, $\TB_n$ denotes those \tough\ vertices whose families have size at least $\cc$ (are ``large'').  Of \tough\ vertices with smaller families, $\TE_n$ 
contains those that have expanding families and $\TG_n:=\TT_n\setminus(\TB_n\cup \TE_n)$ contains the rest. Thus we have the following partitions (see Figure~\ref{fig1}):
 \begin{eqnarray*}
  T_n \cup U &=& \TM_n \cup \TT_n,\\ 
  \TT_n &=&  \TB_n \cup \TE_n \cup \TG_n.
 \end{eqnarray*} 
\begin{figure}[t]
\begin{center}\includegraphics[height=4cm]{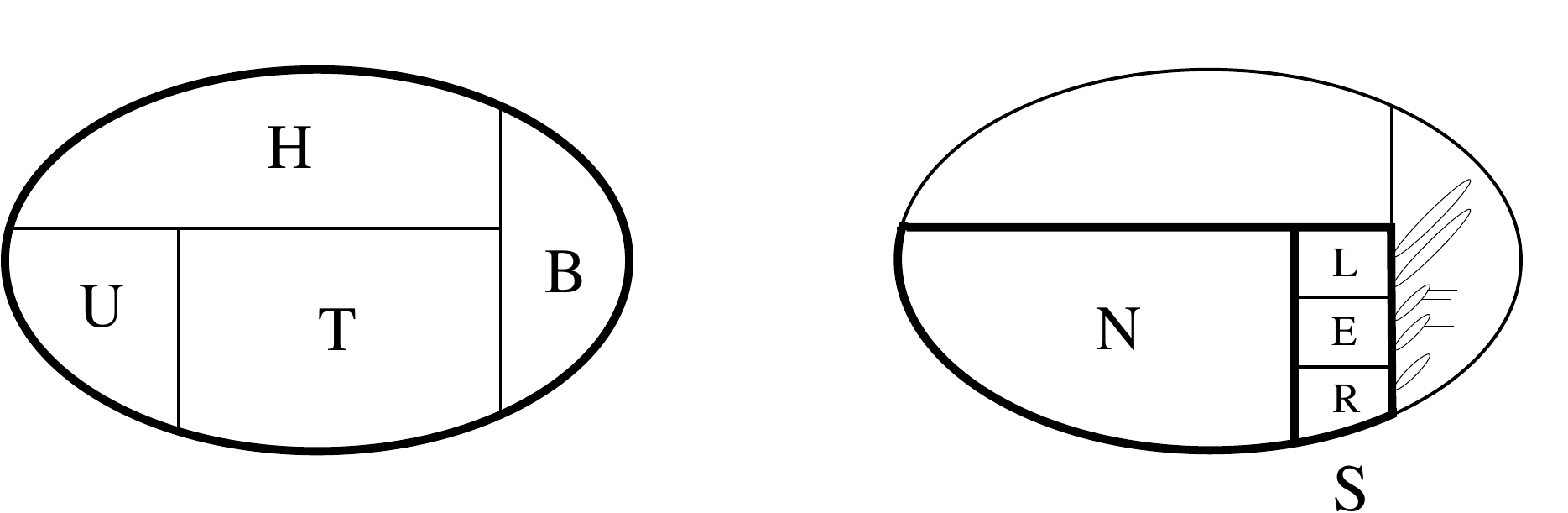}\end{center}
\caption{The structure of the set $X$.}\label{fig1}
\end{figure}
We shall often omit the index $n$ from our notation, except where this may lead to confusion. 

Consider a \tough\ vertex $x \in \TG$ (whose family is small and not expanding). 
Note that $\{x\} \cup F(x)$ consists of at least two vertices
by Claim~\ref{claim:tough_edge}, has at most one vertex of degree $d+1$ 
(since $\{x\}\cup F(x)$ spans a connected subgraph with at most $\cc<r_2$ vertices)
and each its vertex has degree at least $d$ (as it belongs to
$\tilde{T}_n \cup U$). 
Since $\GG$ has no stumps, the number of edges leaving 
$\{x\} \cup F(x)$ is at least $d$. So, if $k$ of these are adjacent to $F(x)$, then at least $d-k$ are adjacent
to $x$ and we have $|E(x,F(x))|\le \deg(x)-(d-k)=k+\I 1_J(x)$, where $\I 1_J$ is the characteristic function of $J$.
The set $F(x)$ cannot send any edges to $B\setminus F(x)$ because it a non-expanding family
nor any edges to $\TT\setminus\{x\}$  by Claim~\ref{cl:onlyonetough}. Hence the edges from
$F(x)$ have to go to $H$, $\TM$ or the outside world $O = X^c$. This gives the following edge count: 
\begin{equation}\label{eq:xFx} |E(F(x), \TG)| = |E(F(x),x)| \leq |E(F(x), H \cup \TM \cup O)| + \I 1_J(x).\end{equation}
 By Claim~\ref{claim:tough_edge} we see that any edge between $\TG$ and $B$ has to run between a 
vertex in $\TG$ and a member of its family. Thus, by integrating~(\ref{eq:xFx}) over $x\in \TG$ and using that families are pairwise disjoint subsets in $B$, we get that
\[ \muCNe{B}{\TG} \leq \muCNe{B}{H \cup \TM \cup O} + \muN{J\cap \TG}. \]
  (Recall that $\muCNe{X}{Y}$, as defined in (\ref{eq:muCe}), denotes the measure of edges between 
$X,Y\subseteq V$.)

We bound the number of edges between any other \tough\ vertex $x\in \TB\cup \TE$ and $B$ by the trivial bound $d$. (Note that if $\deg(x)=d+1$ then $x\in J$ is covered by the
current matching $M$, so at least one
edge at $x$ does not go to $B$.) Adding this to the previous equation, we get
\begin{equation} \muCNe{B}{\TT}  \leq d\,\muN{\TB} + d\,\muN{\TE} + \muCNe{B}{H \cup \TM  \cup O} + \muN{J\cap \TG}.
  \label{TT_eq} \end{equation}

We have that $\muN{T} = \muN{H}$ by Proposition~\ref{prop:prop1}.2 (namely, because $M$ gives
a bijection between these two sets). Also, all vertices in $T$ have degree at least $d$. Thus 
\[
\muCNe{H}{V} - \muN{H\cap J}\le d\,\muN{H}  = d\,\muN{T} = \muCNe{T}{V} - \muN{J \cap T}. 
\]
 Similarly as in Section~\ref{NoOddCycles}, if we choose $r_2 > 16/c_0$ then
to any vertex $x \in H \cap J$ we can associate a unique path of length $r_2/4$ that lies in $X$ and
conclude that  the measure of $H \cap J$ is at most $c_0\,\muN{X}/4$. Hence  
 \[
 \muCNe{H}{V}\le \muCNe{T\cup U}{V}  - \muN{J \cap T} + \frac{c_0}4\, \muN{X}.
 \] 
The edges between $T \cup U$ and $H$ contribute equally to the total degrees of these two sets. In the worst case there are no internal edges in $H$. This boils down to the following estimate:
\begin{multline*}
 \muCNe{H}{O} + \muCNe{B}{H} 
   \ \leq\ 2\,\muCNe{T\cup U}{T\cup U}+ \muCNe{T \cup U}{O}\\
   + \muCNe{B}{\TT} + \muCNe{B}{\TM} - \muN{J\cap T} + \frac{c_0}4\,\muN{X}.
\end{multline*}
Combining this with (\ref{TT_eq}) and subtracting $\muCNe{B}{H}$ from both sides, we get
\begin{multline*} 
\muCNe{H}{O} 
 \ \leq\ 2\,\muCNe{T\cup U}{T\cup U} + \muCNe{B \cup T \cup U}{O}  \\
  + 2\,\muCNe{B}{\TM} + d\,\muN{\TB} + d\,\muN{\TE} 
+ \muN{J\cap \TG} -  \muN{J\cap T} + \frac{c_0}4\, \muN{X}. \end{multline*}
Clearly $\muN{J\cap\TG} \leq \muN{J\cap T}$. Each vertex of $\TB$ has a family of size at least $\cc$, and these 
families
are contained in $B$ by Claim~\ref{cl:subsetB} and are
disjoint by Claim~\ref{disjointfamilies}. Thus we get that $\muN{\TB} \leq \muN{B}/\cc$. Using this and adding $\muCNe{B\cup T \cup U}{O}$ to both sides, we obtain that
\begin{multline}\label{XO_eq} 
 \muCNe{X}{O} 
 \ \leq\ 2\,\muCNe{T \cup U}{T\cup U} + 2\,\muCNe{B\cup T \cup U}{O} \\ 
 + 2\,\muCNe{B}{\TM}+ \frac{d}{\cc}\, \muN{B} + d\,\muN{\TE} + \frac{c_0}4\,\muN{X}.
\end{multline}
Any vertex in $O_n$ that is adjacent to $B_n \cup T_n \cup U$  is going to be in $X_{n+1}$, hence 
\[ \muCNe{B_n \cup T_n \cup U}{O_n} \leq d\,(\muN{X_{n+1}} - \muN{X_n}).\]
Since there is no augmenting path of length at most $2n+2$, any vertex in $\TM_n$ that is adjacent to an edge coming from $B_n$ will be a part of $B_{n+1}$. Likewise, by Lemma~\ref{TTedge_lemma}, any edge in $E(T_n\cup U,T_n\cup U)$ has to be adjacent to a point in $B_{n+1}\setminus B_n$. This implies that 
\[ 2\,\muCNe{T_n\cup U}{T_n \cup U} + 2\,\muCNe{B_n}{\TM_n}  \leq 2d\,(\muN{B_{n+1}} - \muN{B_n}).\]
Plugging all this into (\ref{XO_eq}) and dividing by $d$, we get
\begin{equation} 
 \frac{\muCNe{X_n}{O_n}}{d} 
 \ \leq\ 2\,(\muN{X_{n+1}} - \muN{X_n}) 
 + 2\,(\muN{B_{n+1}} - \muN{B_n}) + \muN{\TE_n} + \frac{\muN{B_n}}{\cc}+\frac{c_0\, \muN{X_n}}{4d}. \label{XO2_eq} 
 \end{equation}
By Definition~\ref{expanding_def},  we have that $e_{n+1}(x) = e_n(x) + 1$ for $x \in \TE_n$
while $e_{n+1}(x)=e_n(x)$ otherwise. Thus 
\[ 
 \int_{X_{n+1}} e_{n+1}(x)\dd x = \int_{X_n} e_n(x)\dd x + \muN{\TE_n}.
\]
Hence the right hand side of (\ref{XO2_eq}) is at most $2(I(n+1)-I(n)) + \muN{B_n}/\cc + c_0\,\muN{X_n}/(4d)$. Furthermore, by Lemmas~\ref{lemma:denseset} and~\ref{lemma:expansion} we have 
 \[ 
 \muCNe{X_n}{O_n} \geq  c_0\, \muN{X_n}, 
 \] 
 which implies that
\[
\frac{c_0\, \muN{X_n}}{2d} \leq I(n+1) - I(n) + \frac{\muN{B_n}}{2\cc} + \frac{c_0\,\muN{X_n}}{8d}.
\]

Recall that $\cc = 2d/c_0$. Since $\muN{B_n} \leq \muN{X_n}$, we get that
\[\frac{\muN{X_n}}{4\cc} \leq I(n+1) - I(n).\]
On the other hand, we know from Claim~\ref{cl:bounded_en_corol} that $e_n(x) \leq \cc^2$ for every
$x\in V$. 
 Thus $\int_{X_n} e_n(x) \dd x \leq \cc^2\,\muN{X_n}$ and
 \begin{equation}\label{eq:IX}
 I(n) \leq \left(2+ \frac{\cc^2}{2}\right)\,\muN{X_n} \leq  \cc^2\, \muN{X_n} 
 \leq 4\cc^3 \big(I(n+1) - I(n)\big).
 \end{equation}
 This gives that $\big(1+1/(4\cc^{3})\big) I(n) \le I(n+1)$. We conclude by induction
on $n$ that
\[ 
 \left(1+\frac{1} {4\cc^3}\right)^nI(0) \le I(n)\le \cc^2\,\muN{X_n}\le \cc^2,
 \]
 as long as there are no augmenting paths of length at most $2n+2$.
Since $X_0=U$, we have that $I(0)=\muN{U}$.
In particular, taking $n=\lfloor (n_0-2)/2\rfloor$, we obtain the desired exponential bound on $\muN{U}$.
This finishes the proof of Theorem~\ref{shortalternating_theorem}.


\section{Proof of Lemma~\ref{pre-colored}}\label{se:pre-colored}

Recall that at most $d$ leaves in a finite graph $G$ with $\Delta(G)\le d$
are pre-colored and and we have to show that this pre-coloring 
can be extended to the whole graph.

Assume that the colors on the pre-colored leaves form a subset of $[d]$. Let $L$ consist of
vertices of degree 1 whose (unique) incident edge is pre-colored. 
Let $Y$ consist of those vertices of $G$ that are incident to at least $\sqrt d$ pre-colored edges.  Clearly, $|Y|\le d /\sqrt d= \sqrt d$. Pick any set of $|Y|$ unused colors from 
$[s]$ and edge-color $G[Y]$ using these colors by Vizing's theorem, where 
$s:=\lfloor d+2\sqrt d\,\rfloor$.

Next, let us color, one by one, 
all uncolored edges that connect $Y$ to $Z:=V(G)\setminus (Y\cup L)$ by using colors from $[s]$ only. 
When we consider a new
edge connecting $y\in Y$ to $z\in Z$ then we have at most $d-1$ colors forbidden at $y$ and at most
$2\sqrt d-1$ colors forbidden at $z$. (Indeed, $z\not\in Y$ 
is incident to at most $\sqrt d$ pre-colored leaves and to at most $|Y|-1$ other colored edges.)
Thus the number of forbidden colors at $\{y,z\}$ is at most $s-1$, so we can extend our coloring to $\{y,z\}$
using some color from $[s]$.

Thus it remains to color the edges in $H:=G[Z]$, the subgraph induced by $Z$. 
By Vizing's theorem, we can find a proper edge-coloring $g:E(H)\to [d+1]$ of the graph $H$. Let $H_g$ 
be a subgraph of $H$ that consists of \emph{$g$-conflicting} edges, i.e.\ those edges inside $Z$ 
that are adjacent to another edge of $G$ of the same color. (Clearly, the latter edge must have the other vertex in~$L\cup Y$.)

We try to ``improve'' the coloring $g$ by composing it with a permutation
$\sigma:[d+1]\to [d+1]$, chosen uniformly at random.  Take a vertex $z\in Z$. There are at most $|Y|+\sqrt d\le 2\sqrt d$ edges in $E(G)\setminus E(H)$ incident to $z$ and
each of these is responsible for at most one conflicting edge at~$z$. Next, consider the random variable $X_z=X_z(\sigma)$ which is the
number of neighbors $x\in Z$ of $z$ such that $\sigma(g(xz))$ is equal to the color at some edge 
between $x$ and $L\cup Y$. In other words, $X_z(\sigma)$ counts the number of 
$H_{\sigma\circ g}$-edges at $z$ with a conflict at the other endpoint. As we 
argued
before,
each $x\in Z$ sends at most $2\sqrt d$ edges to $L\cup Y$. By the linearity of expectation
we have that
 \begin{equation}\label{eq:EXz}
 \mathbf{E}(X_z)\le \deg(z)\, \frac{2\sqrt d}{d+1}<2\sqrt d.
 \end{equation}

Note that $X_z$ changes at most by 2 if we transpose some two elements 
of~$\sigma$. Also, if $X_z(\sigma)\ge i$, then there are
$i$ values of $\sigma$ such that $X_z(\sigma')\ge i$ for every $\sigma'$ that coincides with $\sigma$ on these $i$ values. (Namely, fix the colors of some $i$ conflicting edges at $z$.) Thus all
assumptions of McDiarmid's concentration result~\cite[Theorem 1.1]{mcdiarmid:02} are satisfied (with $c=2$ and $r=1$
in his notation) and we conclude
that, for each $t\ge 0$, the probability of $X_z\ge m+t$ satisfies
 \begin{equation}\label{McD}
 \mathbf{Pr}(X_z\ge m+t)\le 2 \exp\left(-\frac{t^2}{64(m+t)}\right),
 \end{equation}
 where $m$ is the median of $X_z$. Since $X_z$ is non-negative, we have that $\mathbf{E}(X_z)\ge \frac12\, m$. Thus $m<4\sqrt d$ by \eqref{eq:EXz}.  Taking, for example, $t=0.5\sqrt d$ in (\ref{McD}) we obtain
 $$
 \mathbf{Pr}(X_z\ge 4.5\sqrt d)\le 2 \exp\left(-\frac{d/4}{64\cdot 4.5\sqrt d}\right)=\exp(-\Omega(\sqrt d)).
 $$
 The Union Bound shows  that there is $\sigma$ such that $X_z<4.5\sqrt d$ for every vertex $z\in Z$ at distance at most
$2$ from $L \cup Y$. (Note that there are at most $O(d^{5/2})$ such vertices $z$.)
Since all $(\sigma\circ g)$-conflicting edges have to be at distance at most 1 from $L \cup Y$, this permutation
$\sigma$ satisfies that  the $(\sigma\circ g)$-conflict graph $H_{\sigma\circ g}\subseteq H$ has maximum degree at most
$6.5\sqrt d$. Recolor $E(H_{\sigma\circ g})$ with a set of new $\Delta(H_{\sigma\circ g})+1$ colors using Vizing's theorem. 
Clearly, the obtained edge-coloring of $G$ is proper and uses at most $s+6.5\sqrt d+1$ colors, which is at most the stated bound. This finishes the proof of Lemma~\ref{pre-colored}.

\section{An application}\label{MinK}

As we mentioned in the Introduction, a natural question is to determine  $k_{\C B}(d)$ (resp.\ $k'_{\C B}(d)$), the smallest $k$ such that every graphing $\GG=(V,\C B,E,\mu)$ of maximum degree 
$d$ can be generated by $k$ maps $\phi_1,\dots,\phi_k$ as in Definition~\ref{df:graphing} (resp.\ where
we additionally require that $A_i\cap B_i=\emptyset$). The results of Marks~\cite{marks:16} determine these functions exactly.

\begin{propo}\label{pr:kB} We have for all $d\ge 1$ that $k_{\C B}(d)=d$
and $k_{\C B}'(d)=2d-1$.\end{propo}
\begin{proof} Since the case $d=1$ is trivial, assume that $d\ge 2$.
 The lower bound  in both cases can be achieved by the same construction. Namely, take the Borel
graph $\GG=(V,\C B, E)$ constructed by Marks~\cite{marks:16} such that $\Delta(\GG)=d$, 
$\chi_{\C B}'(\GG)=2d-1$ and $\chi_{\C B}(\GG)=2$, with the last property being witnessed by a
partition $V=V_1\cup V_2$.

Not every Borel graph can be made into a graphing by choosing a suitable measure. For example, neither the grandmother graph defined in~\cite[Example 18.36]{lovasz:lngl}
nor any union of its vertex-disjoint copies admits such a measure. However, the Borel graph
constructed by Marks can be turned into a graphing.  In order to
show this, we have to unfold Marks' construction, using~\cite[Lemma~3.12]{marks:16}. 
Namely, let $\Gamma:=\Gamma_1*\Gamma_2$
be the free product of two copies of $\I Z/d\I Z$. 
The group $\Gamma$ naturally acts on $[3]^\Gamma$, the set of functions from $\Gamma$ to $[3]$.
Let $\mathrm{Free}([3]^\Gamma)$ be the \emph{free part} of this action which consists of
those $f\in [3]^\Gamma$ such that  $\gamma\cdot f\not=f$ for all non-identity $\gamma\in\Gamma$.
For $i=1,2$, let $V_i$ consist of \emph{$\Gamma_i$-equivalence classes} of $f\in \mathrm{Free}([3]^\Gamma)$
that is, sets $\{f,x\cdot f,\dots,x^{d-1}\cdot f\}$, where $x$ is a generator of $\Gamma_i$. 
Let  $X_1\in V_1$ and $X_2\in V_2$ be adjacent in $\GG$ if they intersect. Since we restricted ourselves
to the free part, each equivalence class consists of $d$ elements and the obtained graph $\GG$ is $d$-regular. Its vertex set $V=V_1\cup V_2$ admits the natural Borel structure coming from the product 
topology on $[3]^\Gamma$ as well as the natural probability measure $\mu$: to sample from $\mu$ take the
$\Gamma_i$-equivalence class of $f:\Gamma\to[3]$, where the index $i\in [2]$ and all values $f(\gamma)\in [3]$ for $\gamma\in\Gamma$ are uniform and independent. Let us show that we indeed have a graphing.
Note that $f\in \mathrm{Free}([3]^\Gamma)$ with probability 1 and the natural projection $p_i:V_i'\to V_i$ which maps an element of $V_i':=\mathrm{Free}([3]^\Gamma)$ to its $\Gamma_i$-equivalence class is measure-preserving. Let 
$\GG'$ be the bipartite graph on the disjoint union of $V_1'$ and $V_2'$ obtained by pulling $\GG$ back along $p_1\sqcup p_2$ (where each edge of $\GG$ gives $d^2$ edges in $\GG'$). A moment's thought
reveals that $E(\GG')$ can be generated as in \eqref{eq:E} by $d^2$ functions $\phi_{x,y}:V_1'\to V_2'$ for $x\in\Gamma_1$
and $y\in \Gamma_2$, where $\phi_{x,y}$ acts on $f\in V_1'$ first by $x$ and then (viewing the result as an element of $V_2'$)
by $y$. Clearly, each $\phi_{x,y}$
is measure-preserving and thus $\GG'$ is a graphing. It routinely follows that $\GG$ is a graphing too.

Now, if the bipartite graphing $\GG$ can be defined by $k$ Borel maps $\phi_i:A_i\to B_i$, $i=1,\dots,k$, as in 
Definition~\ref{df:graphing}, then its edge-set can be partitioned into $2k$ Borel matchings that are defined inductively on $i=1,\dots,k$ as follows:
 \begin{eqnarray*}
 M_i&:=&\big\{\{x,\phi_i(x)\}\mid x\in V_1\cap A_i\big\}\setminus \cup_{j=1}^{i-1} (M_j\cup M_j'),\\
 M_i'&:=&\Big(\big\{\{x,\phi_i(x)\}\mid x\in V_2\cap A_i\big\}\setminus M_i\Big)\setminus \cup_{j=1}^{i-1} (M_j\cup M_j').
 \end{eqnarray*} 
 Thus $2k\ge \chi_{\C B}'(\GG)=2d-1$, that is, $k\ge d$. If, furthermore, $A_i\cap B_i=\emptyset$ 
for all $i$, then we directly get a partition of $E$ into $k$ Borel matchings as in~\eqref{eq:E}, that is,
$k\ge \chi_{\C B}'(\GG)=2d-1$. This gives the desired lower bounds on $k_{\C B}(d)$ and $k_{\C B}'(d)$.

Conversely, let $\GG=(V,\C B,E,\mu)$ be an arbitrary graphing with maximum degree $d$. 
Proposition~\ref{pr:subgraphing} shows that if $\phi$ is an invertible Borel map between
two Borel subsets $A,B\subseteq V$ such that $\{x,\phi(x)\}\in E$ for all $x\in A$ then $\phi$ 
preserves the measure~$\mu$. In particular, every Borel matching $M\subseteq E$ can be
represented by one such function $\phi$ (by picking one element $x$ in each $\{x,y\}\in M$ in
a Borel way and letting $\phi(x)=y$). Since $E$ can be partitioned into at most $2d-1$ Borel
matchings by Theorem~\ref{th:KST}, we conclude that $k_{\C B}'(d)\le 2d-1$.

Likewise, in order to prove that $k_{\C B}(d)\le d$, let us show that $E$ can be partitioned into at most
$d$ Borel directed graphs $F_1,\dots,F_d$, each with maximum in- and out-degree at most $1$. First, take a 2-sparse labeling $\ell:V\to[m]$. Initially, let each $F_i$ be empty. 
Iteratively, over pairs $\{u,v\}\subseteq [m]$, take all edges of $E$ labeled as $\{u,v\}$ and for each such 
edge $\{x,y\}$ pick the lexicographically smallest triple $(j,\ell(a),\ell(b))$ where $j\in [d]$, $\{a,b\}=\{x,y\}$,
and when we add the ordered arc $(a,b)$ to $F_j$ then both
maximum in-degree and maximum out-degrees of $F_j$ are still at most~$1$. Note that at least one such choice of $(j,a,b)$
exists: if some $j$ is forbidden, then $F_j$ has already at least two arcs connecting 
$\{x,y\}$ to its complement, which rules out  at most $d-1$ values of~$j$.
Also, the choices that we simultaneously make for some pair $\{u,v\}$ cannot conflict with
each other by the 2-sparseness of~$\ell$. Clearly, all sets (and maps) that we obtain are
Borel. This finishes the proof.\qed\end{proof}

It would be fair to say that the question addressed by Proposition~\ref{pr:kB} is more 
about Borel graphs rather than graphings. Indeed, it asks for a Borel decomposition of 
$E$ into matchings (or unions of directed paths and cycles) and the role of the measure $\mu$ in the
definition of $k_{\C B}$ and $k_{\C B}'$ is only to restrict us to those Borel graphs that
can be turned into graphings. The proof of Proposition~\ref{pr:kB} shows that we can drop this
restriction and yet the values of  $k_{\C B}$ and $k_{\C B}'$ will not change. 

On the other hand, one can ignore a set of measure zero in many applications of graphings. 
Note that, modulo removing a null-set of vertices, Definition~\ref{df:graphing} does not change if we require only that the sets $A_i,B_i$ are in $\C B_\mu$, the completion of $\C B$ with respect to $\mu$, while $\phi_i$ is $\mu$-measurable. Indeed, every $\mu$-measurable $\phi_i:A_i\to B_i$ can be made Borel by removing a null-set from $A_i$ (and the corresponding null-set from $B_i$). This change of definition may bring $k$ down. With this in mind, we define $k(d)$ (resp.\ $k'(d)$) as the smallest $k$ such that for every graphing $\GG=(V,\C B,E,\mu)$ with $\Delta(\GG)=d$ there are $k$ invertible measure-preserving maps $\phi_i:A_i\to B_i$ with $A_i,B_i\in\C B_\mu$ for $i=1,\dots,k$ such
that (\ref{eq:E}) holds (resp.\ where we additionally require that $A_i\cap B_i=\emptyset$). Note that the maps $\phi_i$ and $\phi_i^{-1}$ in the definition of $k(d)$ and $k'(d)$ are $\mu$-measurable but not necessarily Borel. 

Interestingly, this relaxation of the restrictions on  $\phi_i$'s reduces the minimum $k$ by factor $2+o(1)$ as $d\to \infty$, which follows with some work from Theorem~\ref{vizing theorem}. We need an auxiliary result
first.

\begin{lemma}\label{lm:mbreak} Let the edge-set of a graphing $\GG=(V,\C B,E,\mu)$ be partitioned
into Borel sets, $E=F_0\cup F_1\cup\cdots\cup F_k$, so that $F_j$ has maximum degree
at most $2$ for each $j\in [k]$  while $F_0$ is a matching. Then  there is a Borel matching $M\subseteq E$ such that the measure of vertices in
infinite components of each $F_j\setminus M$, $j\in [k]$, and of $F_0\cup M$ is zero.\end{lemma}

\begin{proof} 
Choose a fast growing sequence of integers $d_0\ll d_1\ll d_2\ll \dots$~.
Initially, let $M:=\emptyset$. 
We define $F_j':=F_j\setminus M$ for $j\in [k]$ and $F_0':=F_0\cup M$; these are updated every time when the current matching $M$ changes. 
We repeat a certain iteration step over $i\in\I N$. Given $i$, pick some $j\in\{0,\dots,k\}$ so that each $j$ is considered for infinitely many values of~$i$. For example, let us agree
that $j=j(i)$ is always the residue of $i$ modulo $k+1$. 

Informally speaking, given the current $i$ we ``take care'' of $F_j'$ by changing $M$ so that the updated edge-set $F_j'$ has only finite components, each of size at most $O(d_i)$. By doing this carefully, we can ensure that we change $M$ on a set of measure $O(d_i^{-1})$. Of course, some iteration step at a later moment $h>i$ may create infinite components in $F_j'$. But, since each $F_j'$ is ``repaired'' for infinitely many moments $i$, a vertex $y$  belongs to an infinite component of the final set $F_j'$ only if some ``bad'' events that are related to $y$ happen for infinitely many values of~$h$. An application of the Borel--Cantelli Lemma will show that the 
measure of such vertices $y$ is zero. 

Before we can describe the iteration step, we need some definitions. For a set $Y\subseteq E$, let $\partial{Y}:=\{x\in V(Y)\mid \deg_Y(x)=1\}$ consist of  vertices that are incident to exactly one edge of $Y$. A subset $D\subseteq Y$ is called
\emph{$(Y,r)$-sparse} if, for every edge $e\in D$, the distance with respect to $Y$ between $e$
and $V(D\setminus \{e\})\cup \partial{Y}$ is strictly larger than~$r$. 
Informally, $D$ is $(Y,r)$-sparse if, within $Y$, no element of $D$ is close
to another element of $D$ or to the ``boundary'' $\partial{Y}$.

Now, given $i\in \I N$, the corresponding iteration step is as follows. Let $X:=F_j'$
if $j\not =0$ and
$X:=M$ otherwise.  Take a maximal Borel $(F_j',d_i)$-sparse subset 
$D_i\subseteq X$. Such a set $D_i$ 
can be constructed by the familiar argument where we take a $(d_i+1)$-sparse labeling $V(F_j')\to [s]$ of $F_j'$ and, iteratively over all pairs $\{x,y\}\subseteq [s]$, add to $D_i$ all admissible edges from $X$ whose label set is~$\{x,y\}$.)  By definition,
the obtained set $D_i$ is a matching.
If $j=0$, then remove $D_i$ from $M$. If $j\not=0$, then add $D_i$ to $M$ and remove 
  $$
 D_i':=\{e\in M\setminus D_i\mid  e\cap V(D_i)\not=\emptyset\}
 $$
 from $M$. (In order words, we ensure that $M$ is still a matching by removing the set $D_i'\subseteq M$ of  the ``earlier'' edges that
conflict with $D_i$.) Note that, in both cases, the updated set $F_j'$ loses
all edges from $D_i$.  We define the final set $M_\infty\subseteq E$ to consist of those pairs that are eventually included into the current matching from some moment onwards:
 $$
 M_\infty:= \cup_{i\not\equiv 0}\left(  D_i\setminus\left((\cup_{h\ge i\atop h\not\equiv0} D_h') \cup (\cup_{h>i\atop h\equiv0} D_h)\right)\right).
 $$
 (Here and below all residues are modulo $k+1$.)
 Clearly,
$M_\infty$ is a Borel matching. Let us show that it has all required properties.

First, we show that each set $D_i$ has small measure. For convenience, assume that each $d_i$ is even. By construction, $D_i$ is $(F_j',d_i)$-sparse.  Thus the $(d_i/2)$-neighborhoods of edges in $D_i$, taken with respect to $F_j'$,  are pairwise
disjoint and each contains exactly $d_i+2$ vertices as the maximum degree of $F_j'$ is at most~$2$. 
Since all sets are Borel, we conclude by Proposition~\ref{pr:subgraphing} 
that $\muN{V(D_i)}\le 2/(d_i+2)$.

Let $Y_i:=\cup_{h=i+1}^\infty N_{2d_i+4}(V(D_h))$ consist of vertices that belong to the $(2d_i+4)$-neighborhood (taken with respect the whole edge-set $E$)
of $V(D_h)$  for at least one $h>i$. 
Since 
 $$
 \muN{N_{2d_i+4}(V(D_h))}\le (\Delta(\GG)-1)^{2d_i+4}\,\muN{V(D_h)}\le (2k)^{2d_i+4} \cdot \frac{2}{d_h+2},
 $$
 it follows that $\muN{Y_i}\le  (2k)^{2d_i+4} \sum_{h=i+1}^\infty 2/(d_h+2)$. 
By letting the numbers $d_h$ grow 
sufficiently
fast, we can ensure that $\sum_{i=0}^\infty \muN{Y_i}<\infty$. The Borel--Cantelli Lemma implies that
the set 
 $$Y:=\cap_{i=0}^\infty \cup_{h=i+1}^\infty Y_h$$ 
 of vertices that belong to infinitely many of the
sets $Y_i$ has measure zero.

Thus, in order to prove the lemma, it suffices to show that, for every 
$j\in\{0,\dots,k\}$, each vertex $y$ 
of an infinite component of the final set $F_j'$ belongs to $Y$. In fact, we are going to show the stronger claim
that $y\in Y_i$ for every $i\equiv j$.  Fix any such $i$. 

First, consider the case $j\not=0$. Then the final set is $F_j'=F_j\setminus M_\infty$. Consider the
moment when we are about to add $D_i$ to $M$. Recall that
$D_i$ is a maximal $(F_j',d_i)$-sparse subset of the current set $F_j'$. The maximality of $D_i$ implies that if we move
from $y$ in any of at most two possible directions along $F_j'$ (recall that $\Delta(F_j')\le 2$), then
we encounter within $2d_i+3$ steps an element of $D_i$ or a vertex of $F_j'$-degree at most $1$. 
Thus, at the moment right after we added $D_i$ to $M$ and before we removed $D_i'$ from $M$,
the component $C$ of $F_j'$ that contained $y$ was entirely covered by $N_{2d_i+3}(y)$. 
At the end, the $F_j'$-component $C$ of $y$ became
infinite. Consider the first time (after $D_i$ was added) when a new edge $e$ is attached to the current component $C\ni y$
(perhaps after some steps when $C$ had shrunk further).  Let us show 
that this cannot happen when $D_i'$ is removed from $M$. For this, it is enough to show that  $D_i'\cap F_j=\emptyset$. Now, if $D_i'\cap F_j\ni \{u,v\}$ with, say, $\{v,w\}\in D_i$, then $v$ would have
degree $1$ in $F_j'$ and the edge $\{v,w\}\in D_i$ would be too close in $F_j'$ to a degree-$1$
vertex, a contradiction. Thus there are only two ways for the edge $e$ to be added to $F_j'$: for 
some $h>i$ either $e\in D_h'$ with $h\not\equiv 0$
(that is, $e$ is removed from $M$ because
it intersects $V(D_h)$) or $e\in D_h$ with $h\equiv 0$.  Since $\dist(y,e)\le 2d_i+3$,  we conclude 
that $y\in N_{2d_i+4}(V(D_h))$, that is, $y\in Y_i$.

Finally, suppose that $j=0$ (that is, $i\equiv0$).  Then the final set is $F_0'=F_0\cup M_\infty$. At the moment, when we have just 
removed $D_i$ from $M$, the $F_0'$-component $C\ni y$ is a subset of
$N_{2d_i+4}(y)$ by a similar argument as above. (Here, $D_i$ is restricted to a subset
of $M$ but this can increase our distance estimates at most by 1 as $M$ contains every second
edge of each path in the current set $F_0'=F_0\cup M$.)  Again, consider the first
moment when some new edge $e$ gets attached to $C$. 
Here, this can happen in only one possible way, namely, $e$ belongs to some
$D_h$ where $h>i$ is not divisible by $k+1$ (then this set $D_h$ is added to the matching).
Here $y\in N_{2d_i+4}(V(D_h))$ and thus $y\in Y_i$. This completes the proof of the lemma.\qed\end{proof}

\begin{theorem}\label{th:MinK} Let $d\ge 1$. Then 
 \begin{eqnarray*}
 \lceil (d+1)/2\rceil &\le\ k(d)\ \le & \lceil (d+f(d)+1)/2\rceil,\\
 d+1  & \le\ k'(d)\ \le & d+f(d).
 \end{eqnarray*}
 In particular, by Theorem~\ref{th:main}, $k(d)=d/2+o(d)$ and $k'(d)=d+o(d)$ as $d\to\infty$.
\end{theorem}

\begin{proof} 
Let $\GG=(V,\C B,E,\mu)$ be a graphing with maximum degree at most $d$.
By Theorem~\ref{vizing theorem}, there is a partition $E=M_0\cup \dots\cup M_m$ into Borel matchings $M_1,\dots,M_m$ and a $\muC$-null-set $M_0\subseteq E$, where $m:=d+f(d)$. We can additionally assume that $M_0$ is the union  of some connectivity components of~$\GG$. Using the Axiom of Choice and (finite) Vizing's theorem, 
we can partition $M_0$ into $d+1$ matchings $M_1',\dots,M_{d+1}'$. Then the measurable
matchings $M_i\cup M_i'$, $i\in [d+1]$, and $M_i$, $i\in \{d+2,\dots,m\}$, can be oriented (by using
some fixed $1$-sparse labeling) to produce
the measure-preserving maps $\phi_1,\dots,\phi_{m}$ that establish the claimed upper bound $k'(d)\le d+f(d)$.

The upper bound on $k(d)$ follows by pairing the above Borel matchings $M_1,\dots,M_{m-1}$ 
into $k:=\lceil (m-1)/2\rceil$ graphs $F_1,\dots,F_k$
of maximum degree at most $2$ and taking $F_0:=M_m$. Let
$M$ be the matching returned by Lemma~\ref{lm:mbreak}.  We obtain a partition of $E$ a.e.\ into $k+1$ subgraphs, $F_0\cup M,F_1\setminus M,\dots,F_k\setminus M$ of maximum degree at most $2$ with finite components.  The argument of Lemma~\ref{lm:finite} shows that we can orient all these graphs into directed paths 
and cycles in a Borel way. Then we fix the null-set of errors using
the upper bound of $\lceil (d+1)/2\rceil\le k+1$ for countable graphs. This naturally gives $k+1$ measurable maps that generate~$E$. 

The claimed lower bound on  $k'(d)$ is easy: for example, take a finite graph $G$ with $\Delta(G)=d$ and $\chi'(G)=d+1$ and turn it into a graphing
by using the uniform measure on the vertex set.
Also, to show that $k(d)\ge (d+1)/2$ for odd $d$, take any graphing such that the measure of vertices of degree
$d$ is positive and observe that one needs at least $\lceil d/2\rceil=(d+1)/2$ functions $\phi_i$
to represent all edges at a degree-$d$ vertex. 

Finally, the stated lower bound on $k(d)$ for even $d$ can be
obtained by taking the $d$-regular bipartite graphing $\C G$ of Laczkovich \cite{laczkovich:88} and Conley and Kechris~\cite[Section~6]{conley+kechris:13}
that was discussed in Remark~\ref{rm:evend}. It cannot be  represented by
$d/2$ functions $\phi_1,\dots,\phi_{d/2}$ a.e.,\ for otherwise $\C G$ 
would have a perfect matching a.e.\ (for example, $\{\{a,\phi_1(a)\}\mid a\in 
A\}$,
where $A\cup B$ is a bipartition of the vertex set). Note that a finite graph 
would not work here because its edges
can be partitioned into $d/2$ subgraphs of maximum degree at most
2 by Petersen's 2-Factor Theorem~\cite{petersen:91}.
\qed\end{proof}

\begin{remark} One might think that, in the proof of Theorem~\ref{th:MinK}, each $F_i$ could just be oriented without removing any matching $M$. This is however not true, as the following example of a 2-regular graphing $\GG$ shows. Namely, $\GG$ cannot be oriented in a measurable way to have maximum in- and out-degree at most~1. (A different construction of such $\C G$, due to Adams, can be
found in~\cite[Remark~6.8]{kechris+miller:toe}.)

To construct $\GG$, take two
copies of the circle, say $C_j:=\{(e^{2\pi \mathrm{i} x},j): 0\le x< 1\}$ for $j=1,2$, where $\mathrm{i}\in\I C$ is a square root of $-1$. The first measure-preserving map $\phi$ maps $(e^{2\pi\mathrm{i}  x},j)$ to $(e^{2\pi \mathrm{i} x},3-j)$ for $(x,j)\in [0,1)\times[2]$, i.e.\ it is the natural involution between the two circles. The second map $\psi$ has each circle as an invariant set. Namely, for $j\in[2]$,
fix  an axis $A_j$ via the center of the circle $C_j$ and let the restriction of $\phi$ to $C_j$ be the
reflection along $A_j$. Let us assume that $\alpha/\pi$
is irrational where $\alpha$ is the angle between $A_1$ and $A_2$. Suppose on the contrary that we can orient 
the edges of $\GG$ with all in- and out-degrees being $1$ a.e. Let $X$ consist of those $x\in C_1$ such that  the orientation goes from $x$ to $\phi(x)$.
The measure of $X$ is exactly half of measure of $C_1$, because
exactly half of edges in measure between the circles goes each way.
Consider the composition $\phi\circ \psi\circ \phi\circ \psi:C_1\to C_1$, which is a rotation by angle $2\alpha$. It follows that $X$ is invariant a.e.\ with
respect to this irrational rotation of $C_1$, contradicting its ergodicity.
\end{remark}

\section*{Acknowledgments}

The authors thank the anonymous referee for many useful comments.

Endre Cs\'oka was partially supported by ERC
grants~306493 and 648017, and by the  MTA R\'enyi ``Lend\"ulet'' Groups and 
Graphs Research Group.

G\'abor Lippner was partially supported by the  MTA R\'enyi ``Lend\"ulet'' Groups and 
Graphs Research Group.

Oleg Pikhurko was partially supported by ERC
grant~306493 and EPSRC grant~EP/K012045/1.

\renewcommand{\baselinestretch}{1.1}

\small

\providecommand{\bysame}{\leavevmode\hbox to3em{\hrulefill}\thinspace}
\providecommand{\MR}{\relax\ifhmode\unskip\space\fi MR }
\providecommand{\MRhref}[2]{%
  \href{http://www.ams.org/mathscinet-getitem?mr=#1}{#2}
}
\providecommand{\href}[2]{#2}

\end{document}